\documentclass[11pt]{article}      % Comments after  % are ignored^M
\usepackage{amsmath,amssymb,amscd,amsthm, graphicx, verbatim, setspace} % Typical maths resource packages^M
\usepackage{authblk,hyperref}

\theoremstyle{plain}
\newtheorem{thm}{Theorem}[section]
\newtheorem{lem}[thm]{Lemma}
\newtheorem{cor}[thm]{Corollary}

\newtheorem{prop}[thm]{Proposition}
\newtheorem{conj}[thm]{Conjecture}

\newcommand{\ex}{\operatorname{ex}}

\newcommand{\px}{\operatorname{opx}}
\newcommand{\sat}{\operatorname{sat}}
\newcommand{\dist}{\operatorname{dist}}
\newcommand{\osat}{\operatorname{osat}}

\title{Continuous Tur\'{a}n numbers}
\date{}
\author{Jesse Geneson\\
\small\tt geneson@gmail.com
}
\begin{document}
\maketitle

\begin{abstract} % edited 5/16/19 LH
One of the most famous open problems in combinatorics is the Zarankiewicz problem, which asks for the maximum number of ones in an $n \times n$ matrix that has no $s \times t$ submatrix of all ones. The K\H{o}vari-S\'{o}s-Tur\'{a}n theorem provides an upper bound of $O(s^{\frac{1}{t}}n^{2-\frac{1}{t}})$ for this problem for fixed $t \ge 2$, 
which is known to be sharp in some cases. The Zarankiewicz problem is a subproblem of the more general problem of determining the maximum number of ones $\ex(n, M)$ in an $n \times n$ 0-1 matrix that avoids the forbidden 0-1 matrix $M$. 

In this paper, we define a notion of containment and avoidance for subsets of $\mathbb{R}^2$. Then we introduce a new, continuous and super-additive extremal function for subsets $P \subseteq \mathbb{R}^2$ called $\px(n, P)$, which is the supremum of $\mu_2(S)$ over all open $P$-free subsets $S \subseteq [0, n]^2$, where $\mu_2(S)$ denotes the Lebesgue measure of $S$ in $\mathbb{R}^2$. We show that $\px(n, P)$ fully encompasses $\ex(n, M)$ up to a constant factor. More specifically, we define a natural correspondence between finite subsets $P \subseteq \mathbb{R}^2$ and 0-1 matrices $M_P$, and we prove that $\px(n, P) = \Theta(\ex(n, M_P))$ for all finite subsets $P \subseteq \mathbb{R}^2$, where the constants in the bounds depend only on the distances between the points in $P$. 

We also discuss bounded infinite subsets $P$ for which $\px(n, P)$ grows faster than $\ex(n, M)$ for all fixed 0-1 matrices $M$. In particular, we show that $\px(n, P) = \Theta(n^{2})$ for any open subset $P \subseteq \mathbb{R}^2$. We prove an even stronger result, that if $Q_P$ is the set of points with rational coordinates in any open subset $P \subseteq \mathbb{R}^2$, then $\px(n, Q_P) = \Theta(n^2)$. Finally, we obtain a strengthening of the K\H{o}vari-S\'{o}s-Tur\'{a}n theorem that applies to infinite subsets of $\mathbb{R}^2$. Specifically, for subsets $P_{s, t, c} \subseteq \mathbb{R}^2$ consisting of $t$ horizontal line segments of length $s$ with left endpoints on the same vertical line with consecutive segments a distance of $c$ apart, we prove that $\px(n, P_{s, t,c}) = O(s^{\frac{1}{t}}n^{2-\frac{1}{t}})$, where the constant in the bound depends on $t$ and $c$. When $t = 2$, we show that this bound is sharp up to a constant factor that depends on $c$. We also extend $\px(n, P)$ to any number of dimensions, and we generalize most of our results including the strengthening of the K\H{o}vari-S\'{o}s-Tur\'{a}n theorem to any number of dimensions.
\end{abstract}

\section{Introduction}
We say that a 0-1 matrix $A$ \emph{contains} a 0-1 matrix $B$ if some submatrix of $A$ can be turned into $B$ by changing some number of ones to zeroes. Otherwise $A$ \emph{avoids} $B$, i.e. $A$ is \emph{$B$-free}. For any 0-1 matrix $M$ and $n \in \mathbb{Z}^{+}$, define $\ex(n, M)$ to be the maximum number of ones in an $M$-free $n \times n$ 0-1 matrix. This extremal function has been investigated since at least seventy years ago, when Zarankiewicz posed the problem of finding $\ex(n, M)$ for matrices $M$ of all ones \cite{zp}. The famous K\H{o}vari-S\'{o}s-Tur\'{a}n theorem gives a partial solution to this problem. 

In addition to the longstanding open Zarankiewicz problem, research on $\ex(n, M)$ has also focused on other well-known classes of 0-1 matrices like permutation matrices. In particular, Marcus and Tardos \cite{MT} showed that every permutation matrix $M$ has $\ex(n, M) = O(n)$, and used this fact to prove the Stanley-Wilf conjecture using results from \cite{FH} and \cite{k}.  F\"{u}redi used the extremal function $\ex(n, M)$ to obtain the sharpest known upper bound of $O(n \log{n})$ on the maximum number of unit distances in a convex $n$-gon \cite{furu}, while Mitchell applied $\ex(n, M)$ to bound the complexity of an algorithm for path minimization in a rectlinear grid with obstacles \cite{mitchell}.

In this paper, we present a new extremal function for subsets of $\mathbb{R}^2$, which fully encompasses $\ex(n, M)$ up to a constant factor. We show that all past results about $\ex(n, M)$ can be translated into results about this new extremal function. In particular, any sharp bounds on $\ex(n, M)$ for a given 0-1 matrix $M$ imply the same sharp bounds on $\px(n, P)$ up to a constant factor for a set of points $P \subseteq \mathbb{R}^2$ that corresponds to $M$. Moreover for any finite set of points $P$, we show that there is a matrix $M$ for which $\px(n, P) = \Theta(\ex(n, M))$. We also consider infinite sets of points, like a single line segment, or two horizontal line segments with left endpoints in the same column and right endpoints in the same column, and more generally vertical stacks of any number of horizontal line segments. We show that these forbidden stacks of horizontal line segments have extremal functions that behave like the Tur\'an numbers of complete bipartite graphs, and we obtain a strengthening of the K\H{o}vari-S\'{o}s-Tur\'{a}n theorem for these forbidden unions of segments, using the measure-theoretic form of Jensen's inequality and Fubini's theorem.

Before describing these results in more detail, we define the new extremal function. Suppose that $P$ and $S$ are both subsets of $\mathbb{R}^2$. Let $X_P$ be the set of all $x$-coordinates of points in $P$, and let $Y_P$ be the set of all $y$-coordinates of points in $P$. Similarly let $X_S$ be the set of all $x$-coordinates of points in $S$, and let $Y_S$ be the set of all $y$-coordinates of points in $S$. We say that $S$ \emph{contains} $P$ if there exist functions $f_X: X_P \rightarrow X_S$ and $f_Y: Y_P \rightarrow Y_S$ such that all statements below are true: 
\begin{enumerate}
\item $\frac{f_X(x_0)-f_X(x_1)}{x_0-x_1} \ge 1$ for all $x_0, x_1 \in X_P$ with $x_0 > x_1$.
\item $\frac{f_Y(y_0)-f_Y(y_1)}{y_0-y_1} \ge 1$ for all $y_0, y_1 \in Y_P$ with $y_0 > y_1$.
\item For all $(x, y) \in P$, we have $(f_X(x),f_Y(y)) \in S$.
\end{enumerate}

If $S$ does not contain $P$, then $S$ \emph{avoids} $P$, i.e. $S$ is \emph{$P$-free}. For any Lebesgue-measurable subset $S \subseteq \mathbb{R}^d$, let $\mu_d(S)$ denote the $d$-dimensional Lebesgue measure of $S$. For any  subset $P \subseteq \mathbb{R}^2$ and $n \in \mathbb{R}^{+}$, define $\px(n, P)$ as the supremum of $\mu_2(S)$ over all open $P$-free subsets $S \subseteq [0, n]^2$.

For any subset $P \subseteq \mathbb{R}^2$, we say that its \emph{rows} are its maximal subsets with the same $y$-coordinates and its \emph{columns} are its maximal subsets with the same $x$-coordinates. Given any \emph{finite} subset $P \subseteq \mathbb{R}^2$, define $M_P$ to be the 0-1 matrix with the same number of rows and columns as $P$, so that the rows of the matrix $M_P$ correspond to the rows of the subset $P$ in the same order from top to bottom, the columns of $M_P$ correspond to the columns of $P$ in the same order from left to right, and $M_P$ has a one in each entry corresponding to an element of $P$ and a zero in each other entry. We prove that $\px(n, P) = \Theta(\ex(n, M_P))$ for all finite subsets $P \subseteq \mathbb{R}^2$.

Thus $\px(n, P)$ fully encompasses the extremal function $\ex(n, M)$ up to a constant factor, since any 0-1 matrix can be turned into a corresponding finite subset of $\mathbb{R}^2$, with points replacing the ones. However, $\px(n, P)$ also includes problems that have no analogue in 0-1 matrices, as there exist bounded subsets $P \subseteq \mathbb{R}^2$ such that $\ex(n, M) = o(\px(n, P))$ for all 0-1 matrices $M$. In particular we show that $\px(n, P) = \Theta(n^2)$ for every open subset $P \subseteq \mathbb{R}^2$.

We also prove a strengthening of the K\H{o}vari-S\'{o}s-Tur\'{a}n theorem. We show that for every fixed integer $t \ge 2$, if $P_{s, t, c}$ consists of $t$ horizontal segments of length $s$ with all left endpoints in the same column and all consecutive segments a distance of $c$ apart, then $\px(n, P_{s, t, c}) = O(s^{\frac{1}{t}} n^{2-\frac{1}{t}})$, where the constants in the bound depend on $t$ and $c$.  Note that $P_{s, t, c}$ looks like an equal sign ($=$) when $t = 2$ and an equivalence symbol ($\equiv$) when $t = 3$. In the case that $t = 2$, our upper bound is sharp up to a constant factor that depends on $c$ using the result of F\"{u}redi \cite{fur2} that $\ex(n, J_{s, 2}) = \Theta(s^{\frac{1}{2}}n^{\frac{3}{2}})$.

Section \ref{basic_prop} focuses on general properties of $\px(n, P)$. In Section \ref{basic_prop0}, we prove some basic properties of $\px(n, P)$ which are analogous to properties of $\ex(n, M)$. In Section \ref{superadd} we prove that $\px(n, P)$ is continuous and super-additive for all $P \subseteq \mathbb{R}^2$. In Section \ref{pxops}, we prove for subsets $P$ with a rightmost column that $\px(n, P)$ only increases by at most $c n$ when we add a horizontal segment of length $c$ to $P$ with its left endpoint on a point in the rightmost column of $P$. We also prove an upper bound on the increase in $\px(n, P)$ for finite subsets $P \subseteq \mathbb{R}^2$ when they are dilated.

Section \ref{mainresult} focuses on connections between $\px(n, P)$ and $\ex(n, M)$. In Section \ref{mainsub}, we prove that $\px(n, P) = \Theta(\ex(n, M_P))$ for all  subsets $P \subseteq \mathbb{R}^2$. The proof splits into a lower bound, where for each finite  subset $P \subseteq \mathbb{R}^2$ we construct an open subset $S \subseteq [0, n]^2$ to show that $\px(n, P) = \Omega(\ex(n, M_P))$. Then we prove that $\px(n, P) = O(\ex(n, M_P))$, using a transformation from open subsets of $\mathbb{R}^2$  to 0-1 matrices. In Section \ref{01cor}, we discuss some corollaries of the main result of Section \ref{mainsub}. In particular, we show that for every finite subset $P \subseteq \mathbb{R}^2$  there exists a constant $\epsilon > 0$ such that $\px(n, P) = O(n^{2-\epsilon})$. 

Section \ref{infinite_p} focuses on bounded infinite subsets $P$. In Section \ref{bin2}, we show that there are bounded infinite subsets $P \subseteq \mathbb{R}^2$ for which $\px(n, P) = \Theta(n^{2})$. In particular, we show that $\px(n, P) = \Theta(n^2)$ for every open subset $P \subseteq \mathbb{R}^2$, where the constants in our bounds depend on $P$.  We show this by proving an even stronger result, that $\px(n, Q_P) = \Theta(n^2)$ where $Q_P$ is the set of points with rational coordinates in $P \subseteq \mathbb{R}^2$.

In Section \ref{skst}, we prove the strengthening of the K\H{o}vari-S\'{o}s-Tur\'{a}n theorem for vertically stacked unions of horizontal segments. In Section \ref{segperm}, we show that if $P$ is a disjoint union of a finite number of horizontal segments with no two points having the same $x$-coordinate and no two segments having the same $y$-coordinate, then $\px(n, P) = O(n)$.

In Section \ref{higher_d}, we extend the definition of $\px(n, P)$ to $\mathbb{R}^d$ and we generalize most of our results including the strengthening of the K\H{o}vari-S\'{o}s-Tur\'{a}n theorem to any number of dimensions. We use the generalized version of the strengthening of the K\H{o}vari-S\'{o}s-Tur\'{a}n theorem to derive sharp bounds on Tur\'an numbers of forbidden subsets of $\mathbb{R}^3$, including forbidden sets of points that look like two plus signs ($+$) with one directly above the other.

In Section \ref{openpr}, we discuss future directions for research including open problems.

\section{Properties of $\px(n, P)$}\label{basic_prop}

In this section, we prove basic facts about $\px(n, P)$. We start with several observations that have quick proofs. Then we show that $\px(n, P)$ is super-additive and continuous. We also investigate the effect of simple modifications to $P$ on the value of the extremal function $\px(n, P)$. For subsets $P$ with a rightmost column, we prove a sharp upper bound of $c n$ on the increase in $\px(n, P)$ when we add a horizontal segment of length $c$ to $P$ with its left endpoint on some point in the rightmost column of $P$. We also bound the increase in $\px(n, P)$ for finite subsets $P \subseteq \mathbb{R}^2$ when we dilate $P$.

\subsection{Basic observations about $\px(n, P)$}\label{basic_prop0}

It is well-known that if $M'$ is obtained from the 0-1 matrix $M$ by $90^{\circ}$ rotation, or horizontal or vertical reflection, then we have $\ex(n, M') = \ex(n, M)$. If $M' = R(M)$ where $R$ is the transformation that we apply to $M$ to get $M'$, and if $A$ is an $n \times n$ 0-1 matrix that avoids $M$, then $R(A)$ avoids $M'$ and $R(A)$ is still an $n \times n$ 0-1 matrix when $R$ is a $90^{\circ}$ rotation, horizontal reflection, or vertical reflection. Thus we have $\ex(n, M) \le ex(n, M')$. Similarly we obtain $\ex(n, M) \ge ex(n, M')$ since $M = R^{-1}(M')$, so $\ex(n, M) = ex(n, M')$. Below we observe an analogous fact for $\px(n, P)$.

\begin{lem}
If $P' \subseteq \mathbb{R}^2$ is obtained from $P \subseteq \mathbb{R}^2$ by $90^{\circ}$ rotation, or horizontal or vertical reflection, then we have $\px(n, P') = \px(n, P)$.
\end{lem}

\begin{proof}
Suppose that $P' = R(P)$ where $R$ is the transformation that we apply to $P$ to get $P'$, and suppose that $S \subseteq [0,n]^2$ is a $P$-free open subset. Then $R(S)$ avoids $P'$ and $R(S)$ is still an open subset. Moreover $R(S)$ can be translated into a subset of $[0,n]^2$. Thus we have $\px(n, P) \le \px(n, P')$. Similarly we obtain $\px(n, P) \ge \px(n, P')$, so $\px(n, P) = \px(n, P')$. 
\end{proof}

Next we make another simple observation about $\px(n, P)$ which does not have an exact analogue for $\ex(n, M)$.

\begin{lem}
Suppose that $P'$ is obtained by translating $P$. Then $\px(n, P) = \px(n, P')$. 
\end{lem}

\begin{proof}
By definition of containment for subsets of $\mathbb{R}^2$, any subset that contains $P$ will contain $P'$, and any subset that contain $P'$ will contain $P$. Thus $\px(n, P) = \px(n, P')$. 
\end{proof}

It is clear that $\ex(n, M) \le \ex(n, N)$ for any 0-1 matrix $N$ that contains $M$, since any $n \times n$ 0-1 matrix that avoids $M$ will also avoid $N$. We start by observing an analogous fact for $\px(n, P)$.

\begin{lem}\label{pxcontain}
If $P$ and $Q$ are subsets of $\mathbb{R}^2$ for which $Q$ contains $P$, then $\px(n,P) \le \px(n, Q)$.
\end{lem}

\begin{proof}
Any open subset of $[0,n]^2$ that avoids $P$ will also avoid $Q$.
\end{proof}

Another immediate fact is that $\px(n, P)$ is non-decreasing in $n$.

\begin{lem}\label{nondec}
For all $0 < m \le n$, we have $\px(m, P) \le \px(n, P)$.
\end{lem}

\begin{proof}
Any open subset of $[0,m]^2$ that avoids $P$ is also a subset of $[0, n]^2$.
\end{proof}

\subsection{Super-additivity and continuity of $\px(n, P)$}\label{superadd}

In fact, it is possible to prove a much stronger result than the last remark. It is well-known that $ex(n, M)$ is super-additive, i.e., $\ex(m+n, M) \ge \ex(m, M)+\ex(n, M)$ for all $m, n \in \mathbb{Z}^{+}$ \cite{PT}. We prove that $\px(n, P)$ is super-additive for all $P \subseteq \mathbb{R}^2$. 

\begin{lem}
For all $P \subseteq \mathbb{R}^2$, $\px(n, P)$ is super-additive.
\end{lem}

\begin{proof}
If $P$ is unbounded, then the result is immediate since $\px(n, P) = n^2$ for all $n > 0$, so suppose that $P$ is bounded. If $P$ has a single point, then $\px(n, P) = 0$, so suppose that $P$ has multiple points. Without loss of generality, suppose that $P$ has multiple rows.

Let $s_P$ be the supremum of the $y$-coordinates of the rows of $P$, and let $i_P$ be the infimum of the $y$-coordinates of the rows of $P$, which both must exist because $P$ is bounded. Without loss of generality, suppose for all $\epsilon > 0$ that there exist rows $u$ and $r$ of $P$ with $s_P - u < \epsilon$ and $r - i_P < \epsilon$ such that there is some point $p$ in row $u$ that is \emph{not} to the left of some point $p'$ in row $r$. Note that if this supposition was false, then there would exist $\epsilon > 0$ such that for all rows $u$ and $r$ of $P$ with $s_P - u < \epsilon$ and $r - i_P < \epsilon$, every point $p$ in row $u$ is to the left of every point $p'$ in row $r$. Thus if the supposition was false, then for all $\epsilon > 0$ there would exist rows $u$ and $r$ of $P$ with $s_P - u < \epsilon$ and $r - i_P < \epsilon$ such that there is some point $p$ in row $u$ that is \emph{not} to the right of some point $p'$ in row $r$. The only logical difference between this and the supposition is the word \emph{right} replacing the word \emph{left}, so we can assume without loss of generality that for all $\epsilon > 0$ there exist rows $u$ and $r$ of $P$ with $s_P - u < \epsilon$ and $r - i_P < \epsilon$ such that there is some point $p$ in row $u$ that is \emph{not} to the left of some point $p'$ in row $r$. 

Let $S_t \subseteq [0, n]^2$ for $t \in \mathbb{N}$ be a sequence of $P$-free open subsets with the property that $\lim_{t \rightarrow \infty} \mu_2(S_t) = \px(n, P)$. Let  $R_t \subseteq [0, m]^2$ for $t \in \mathbb{N}$ be a sequence of $P$-free open subsets with the property that $\lim_{t \rightarrow \infty} \mu_2(R_t) = \px(m, P)$. For each $t \in \mathbb{N}$, let $Q_t$ be the point set obtained from placing a copy of $S_t$ in the top left corner of $[0, m+n]^2$ and placing a copy of $R_t$ in the bottom right corner of $[0, m+n]^2$. Clearly $Q_t$ is open since it is a union of open sets. Since $S_t$ is an open subset of $[0, n]^2$, it contains no points on the boundary of $[0,n]^2$. Similarly since $R_t$ is an open subset of $[0,m]^2$, it contains no points on the boundary of $[0, m]^2$. Thus the copy of $S_t$ in $Q_t$ is fully above and fully to the left of the copy of $R_t$ in $Q_t$.

Note that $Q_t$ avoids $P$. Indeed, suppose for contradiction that $Q_t$ contained $P$. The copy of $P$ in $Q_t$ cannot be fully contained in the copy of $S_t$ in $Q_t$ or the copy of $R_t$ in $Q_t$, since they both avoid $P$, so parts of the copy of $P$ must be in both $S_t$ and $R_t$. In particular, there exists $\epsilon > 0$ such that any rows $u$ in the copy of $P$ in $Q_t$ with $s_P - u < \epsilon$ must be contained in $S_t$, and any rows $r$ in the copy of $P$ in $Q_t$ with $r - i_P < \epsilon$ must be contained in $R_t$. Thus any points in $P$ in a row $u$ with $s_P - u < \epsilon$ must be to the left of any points in $P$ in a row $r$ with $r - i_P < \epsilon$. This contradicts the fact that for all $\epsilon > 0$, there exist rows $u$ and $r$ of $P$ with $s_P - u < \epsilon$ and $r - i_P < \epsilon$ such that there is some point $p$ in row $u$ that is \emph{not} to the left of some point $p'$ in row $r$. Thus $Q_t$ avoids $P$. So $\px(m+n, P) \ge \lim _{t \rightarrow \infty} \mu_2(Q_t) = \lim_{t \rightarrow \infty} \mu_2(R_t)+\lim_{t \rightarrow \infty} \mu_2(S_t) = \px(m, P)+\px(n, P)$.
\end{proof}

Next we show that $\px(n, P)$ is continuous in $n$ for $n > 0$. In order to prove this, we start with a simple lemma.

\begin{lem}\label{sandwich}
For all $n > 0$ and all $\epsilon$ with $0 < \epsilon < n$, we have $\px(n, P) \le \px(n+\epsilon, P) < \px(n, P)+3n \epsilon$.
\end{lem}

\begin{proof}
The first inequality is immediate by Lemma \ref{nondec}. For the second inequality, observe that any open $P$-free subset $S \subseteq [0, n+\epsilon]^2$ can be transformed into an open $P$-free subset $S' \subseteq [0, n]^2$ by letting $S'$ be the intersection of the open square $(0, n)^2$ with $S$. Then $\mu_2(S') \ge \mu_2(S)-2\epsilon n-\epsilon^2$, so $\mu_2(S) \le \px(n, P)+2n \epsilon+\epsilon^2 < \px(n, P)+ 3n \epsilon$, which implies that $\px(n+\epsilon, P) < \px(n, P)+3n \epsilon$.
\end{proof}

\begin{prop}
The function $\px(n, P)$ is continuous in $n$ for $n > 0$.
\end{prop}

\begin{proof}
Fix $n_0 > 0$ and $\epsilon$ with $0 < \epsilon < n_0$. By Lemma \ref{sandwich}, for $\delta = \min(\frac{n_0}{3},\frac{\epsilon}{3n_0})$ and $n < n_0+\delta$ we have $\px(n, P) \le \px(n_0+\delta, P) < \px(n_0, P)+\epsilon$. Moreover if $n > n_0 - \delta$, then we have $\px(n, P) \ge \px(n_0 - \delta, P) > \px(n_0, P)-3(n_0-\delta) \delta > \px(n_0, P)-\epsilon$. Thus we have shown that for all $n_0 > 0$ and for all $\epsilon > 0$ there exists $\delta > 0$ for which $|\px(n, P)-\px(n_0, P)| < \epsilon$ for all $n$ such that $|n-n_0| < \delta$.
\end{proof}

\subsection{Operations for $\px(n, P)$}\label{pxops}

F\"uredi and Hajnal \cite{FH} proved that if $M$ is a 0-1 matrix with a one in the rightmost column in row $r$, and $M'$ is obtained from $M$ by adding a new column on the right with a single one in row $r$, then $\ex(n, M') \le \ex(n, M)+n$. We prove a similar type of result for $\px(n, P)$. Before we prove this result, we prove a lemma that we will use a few times throughout the paper.

\begin{lem}\label{usefullem}
Let $P$ be a horizontal line segment of length $c$ with closed endpoints, and suppose that $S \subseteq [0, n]^2$ is open and $P$-free. For each $y \in [0, n]$, let $S_y$ denote the set of points $(a, b) \in S$ such that $b = y$. Then $\mu_1(S_y) \le c$ for all $y \in [0, n]$. 
\end{lem}

\begin{proof}
Suppose for contradiction that there exists $y \in [0, n]$ such that $\mu_1(S_y) > c$. Since $S$ is open, we have $0 < y < n$ and $S_y$ is a countable disjoint union of intervals $I_1, I_2, \dots$ with open endpoints. Then $\mu_1(S_y) = \sum_{j \ge 1} |I_j|$. Thus for all $\epsilon > 0$, there exists $N_{\epsilon}$ such that $\sum_{j = 1}^{N_{\epsilon}} |I_j| > \mu_1(S_y)-\epsilon$.

Since we are supposing that $\mu_1(S_y) > c$, we must have $\mu_1(S_y) = c+q$ for some $q > 0$. Let $\epsilon = \frac{q}{2}$. Let $N_{\epsilon}$ be sufficiently large so that $\sum_{j = 1}^{N_{\epsilon}} |I_j| > \mu_1(S_y)-\epsilon$. In each open interval $I_j$ for $j = 1, \dots, N_{\epsilon}$, we take a closed interval $C_j$ of length $\max(0,|I_j|-\frac{\epsilon}{2^j})$. Then $\sum_{j = 1}^{N_{\epsilon}} |C_j| = \sum_{j = 1}^{N_{\epsilon}} \max(0,|I_j|-\frac{\epsilon}{2^j}) >  (\mu_1(S_y)-\epsilon)-\epsilon = \mu_1(S_y)-q = c$. Thus the disjoint union of the $C_j$ for $j = 1, \dots, N_{\epsilon}$ must contain $P$, so $S$ contains $P$, which gives a contradiction. Thus $\mu_1(S_y) \le c$ for all $y \in [0, n]$.
\end{proof}

Now we are ready to strengthen the result of F\"uredi and Hajnal from \cite{FH}.

\begin{lem}\label{addedseg}
If $P \subseteq \mathbb{R}^2$ has a rightmost column, and $P'$ is obtained from $P$ by adding a horizontal segment of length $c$ to $P$ with its left endpoint on a point in the rightmost column of $P$, then $\px(n, P') \le \px(n, P)+c n$.
\end{lem}

\begin{proof}
Let $S'_t \subseteq [0, n]^2$ for $t \in \mathbb{N}$ be a sequence of $P'$-free open subsets with the property that $\lim_{t \rightarrow \infty} \mu_2(S'_t) = \px(n, P')$. For each point $p$ in row $r$ of $S'_t$, let $Z(t, p)$ be the set of points in row $r$ of $S'_t$ to the right of $p$. For each $t$, let $S_t$ be obtained from $S'_t$ by removing from each row $r$ of $S'_t$ any points $p$ such that $\mu_1(Z(t, p)) \le c$. In other words, $S_t$ only includes the points $p$ from $S'_t$ for which $\mu_1(Z(t, p)) > c$. Then $S_t$ must avoid $P$, or else $S'_t$ would contain $P'$ by Lemma \ref{usefullem}. 

Also $S_t$ is open. To see this, fix any $p \in S_t$ from some row $r$ and column $g$ of $S'_t$. By definition of $S_t$, $\mu_1(Z(t, p)) > c$, so $\mu_1(Z(t, p)) = c+q$ for some $q > 0$. Since $S'_t$ is open, we can write $Z(t, p)$ as a countable union of disjoint row-$r$ intervals $I_j$ for $j \ge 1$ (which are open in the restriction to row $r$) such that $\sum_{j} |I_j| = c+q$. In each interval $I_j$, we choose a closed interval $C_j \subseteq I_j$ such that $\sum_j |C_j| \ge c+\frac{q}{2}$. Note that we can do this, e.g., by choosing $|C_j| = \max(0, |I_j|-\frac{q}{2^{j+1}})$ for each $j \ge 1$.

Since $S'_t$ is open, around every point $p' \in S'_t$ there is an open ball $B(p', \nu_{p'}) \subseteq S'_t$ of radius $\nu_{p'}$ centered at $p'$ for some $\nu_{p'} > 0$. For each closed interval $C_j$, $\cup_{p' \in C_j} B(p', \nu_{p'})$ is an open cover of $C_j$. Since $C_j$ is closed and bounded, it is compact, so the open cover $\cup_{p' \in C_j} B(p', \nu_{p'})$ has a finite subcover $\cup_{\pi \in K} B(\pi, \nu_{\pi})$ for some finite subset $K \subseteq C_j$. Since $C_j$ is covered by a finite union of open balls $\cup_{\pi \in K} B(\pi, \nu_{\pi})$, there exists $h_j$ such that $B(p', h_j) \subseteq S'_t$ for all $p' \in C_j$.

Since $\sum_j |C_j| \ge c+\frac{q}{2}$, there exists a minimum $N$ such that $\sum_{j = 1}^{N} |C_j| > c+\frac{q}{4}$. Let $H = \min(\left\{h_j : 1 \leq j \leq N \right\})$. Then for every point $p' \in C_j$ for each $j = 1, \dots, N$, we have $B(p',H) \subseteq S'_t$. 

Since there is an open ball $B(p, \nu_p) \subseteq S'_t$, all points $\tau$ in the same column $g$ as $p$ in $S'_t$ with $\dist(\tau, p) \le \frac{\nu_p}{2}$ must satisfy $\tau \in S'_t$. Moreover since the open ball $B(p',H)$ is a subset of $S'_t$ for every point $p' \in C_j$ for each $j$ between $1$ and $N$ inclusive, all points $\tau$ in column $g$ in $S'_t$ with $\dist(\tau, p) \le \min(\frac{\nu_p}{2}, \frac{H}{2})$ must be in $S_t$. This is because if $\tau$ is in row $r'$ and column $g$ of $S'_t$ with $|r-r'| \le \min(\frac{\nu_p}{2}, \frac{H}{2})$, then $\mu_1(Z(t, \tau)) \ge \sum_{j = 1}^{N} |C_j| > c+\frac{q}{4}$.

Let $R = \min(\frac{\nu_p}{2}, \frac{H}{2},\frac{q}{4})$. Then the open ball $B(p, R)$ is a subset of $S'_t$. The points in the intersection of $g$ with $B(p, R)$ must be in $S_t$, as explained in the previous paragraph. Thus all points in $B(p, R)$ to the left of column $g$ must also be in $S_t$, by definition of $S_t$. All points $\tau$ in the intersection of $g$ with $B(p, R)$ must be in some row $r'$ with $|r-r'| < R$ and satisfy $\mu_1(Z(t, \tau)) > c+\frac{q}{4}$. Thus, all points $\tau'$ in $B(p, R)$ to the right of column $g$ in row $r'$ must satisfy $\mu_1(Z(t, \tau')) > c$ by definition of $R$. So all points in $B(p, R)$ to the right of column $g$ must also be in $S_t$, by definition of $S_t$. We showed that $B(p, R) \subseteq S_t$, so $S_t$ is open.

Let $A = S'_t - S_t$, and let $A_y$ denote the set of elements of $A$ in row $y$. By Fubini's theorem and the definition of $S_t$, $\mu_2(A) = \int_{y \in [0, n]} \mu_1(A_y) d\mu_1(y) \le c n$. Thus $\mu_2(S_t) \ge \mu_2(S'_t)- c n$, so $\lim_{t \rightarrow \infty} \mu_2(S_t) \ge \px(n, P') - c n$, which implies that $\px(n, P') \le \px(n, P)+c n$.
\end{proof}

Combining Lemmas \ref{pxcontain} and \ref{addedseg}, we obtain the following corollary about adding a new point to the right of a set of points in the plane.

\begin{cor}\label{addedpt}
If $P$ is a subset of $\mathbb{R}^2$ with a rightmost column, and $P'$ is obtained from $P$ by adding a new point to $P$ that is $c$ to the right of a point $p$ in the rightmost column of $P$, then $\px(n, P') \le \px(n, P)+c n$.
\end{cor}

\begin{proof}
Let $P''$ be obtained from $P$ by adding a horizontal segment of length $c$ to $P$ with its left endpoint on point $p$ in the rightmost column of $P$. Then $P''$ contains $P'$, so $\px(n, P') \le \px(n, P'')$ by Lemma \ref{pxcontain}. Thus by Lemma \ref{addedseg}, we have $\px(n, P') \le \px(n, P'') \le \px(n, P) + c n$.
\end{proof}

Both of the last two results are sharp. For example, $P$ could be a single point. $P'$ would be a horizontal segment of length $c$ in Lemma \ref{addedseg} or a pair of points in the same row at a distance of $c$ in Corollary \ref{addedpt}. We discuss horizontal segments more in Section \ref{infinite_p}.

If $M$ is a 0-1 matrix, let $S(M, k)$ be the 0-1 matrix obtained from $M$ by inserting $k$ rows of zeroes between every consecutive pair of rows of $M$ and $k$ columns of zeroes between every consecutive pair of columns of $M$. Tardos proved that $\ex(n, S(M, k)) = O(ex(n, M)+n)$ for all 0-1 matrices $M$ \cite{tardos05}, where the constant in the bound depends on $k$. Using the same proof, we show a stronger but more restricted bound that we use in our proof of the main result in Section \ref{mainresult}.

\begin{lem} \label{tardos_blank}
For all finite subsets $P \subseteq \mathbb{R}^2$, $\ex(n, S(M_P,k)) \le (k+1)^2 \ex(\lceil \frac{n}{k+1} \rceil, M_P)$.
\end{lem}

\begin{proof}
Suppose that $A$ is an $n \times n$ 0-1 matrix with $\ex(n, S(M_P,k))$ ones that avoids $S(M_P, k)$. Let $A'$ be the $(k+1)\lceil \frac{n}{k+1} \rceil \times (k+1) \lceil \frac{n}{k+1} \rceil$ 0-1 matrix obtained from $A$ by adding at most $k$ rows of zeroes to the bottom of $A$ and at most $k$ columns of zeroes to the right side of $A$. Note that $A'$ must still avoid $S(M_P, k)$, since $A$ avoids $S(M_P, k)$. For each $1 \le x \le k+1$ and $1 \le y \le k+1$, let $A_{x,y}$ be the $\lceil \frac{n}{k+1} \rceil \times \lceil \frac{n}{k+1} \rceil$ 0-1 matrix obtained from $A'$ by taking the submatrix of rows $r$ for which $(k+1)$ divides $ (r-x)$ and columns $s$ for which $(k+1)$ divides $ (s-y)$. For all $x$ and $y$, $A_{x, y}$ must avoid $M_P$ or else $A$ would contain $S(M_P, k)$. Thus each 0-1 matrix $A_{x, y}$ has at most $\ex(\lceil \frac{n}{k+1} \rceil, M_P)$ ones, so $A$ has at most $(k+1)^2 \ex(\lceil \frac{n}{k+1} \rceil, M_P)$ ones.
\end{proof}

The next result is analogous to the last one, but for finite subsets of $\mathbb{R}^2$ instead of 0-1 matrices. It shows that even if we add empty intervals of rows and columns to a finite subset $P$, it only changes $\px(n, P)$ by at most a constant factor (which depends on the size of the added intervals and the distances between the points in $P$). We use this result multiple times in the remainder of the paper.

\begin{lem} \label{ps_blank}
Suppose that $P \subseteq \mathbb{R}^2$ is a finite subset in which the distances between all consecutive pairs of rows and columns are in $[c, d]$. Let $P'$ be a dilation of $P$ by a factor of $q > 1$, i.e. for every point $(x, y) \in P$, the point $(qx, qy) \in P'$. Then $\px(n, P') \le (\lceil\frac{q d}{c}\rceil+1)^2 \px(\lceil \frac{\lceil \frac{n}{c} \rceil}{\lceil\frac{q d}{c}\rceil+1} \rceil c, P)$.
\end{lem}

\begin{proof}
Let $S_t \subseteq [0, n]^2$ for $t \in \mathbb{N}$ be a sequence of $P'$-free open subsets with the property that $\lim_{t \rightarrow \infty} \mu_2(S_t) = \px(n, P')$. We partition $[0, \lceil \frac{n}{c} \rceil c]^2$ into $c \times c$ open squares, deleting any points at the boundaries, and we use the $c \times c$ squares to define a family of subsets. Note that any deleted points at the boundaries have a total measure of $0$. For each $1 \le u \le \lceil\frac{q d}{c}\rceil+1$ and $1 \le v \le \lceil\frac{q d}{c}\rceil+1$, let $[S_t]_{u,v}$ be the subset of $[0,\lceil \frac{\lceil \frac{n}{c} \rceil}{\lceil\frac{q d}{c}\rceil+1} \rceil c]^2$ obtained from $S_t$ through the following operations. 

For each $1 \le u \le \lceil\frac{q d}{c}\rceil+1$, let $J_u$ be the set of all $x$ for which $\lceil\frac{q d}{c}\rceil+1$ divides $\lceil \frac{x}{c} \rceil-u$ excluding $x$ that are multiples of $c$. In $[S_t]'_{u,v}$, include the points in $S_t$ with $x$-coordinates in $J_u$ and $y$-coordinates in $J_v$ (note that these points form open $c \times c$ squares), do not include any points with $x$-coordinates not in $J_u$ or $y$-coordinates not in $J_v$. We form $[S_t]_{u,v}$ from $[S_t]'_{u,v}$ by completely contracting any rows or columns of $c \times c$ squares that were not included in $[S_t]'_{u,v}$ and translating the resulting subset into the square $[0,\lceil \frac{\lceil \frac{n}{c} \rceil}{\lceil\frac{q d}{c}\rceil+1} \rceil c]^2$. 

For all $u$ and $v$, $[S_t]_{u, v}$ is open since it is obtained by translating disconnected components of the intersection of an open set with an open set. Moreover $[S_t]_{u, v}$ must avoid $P$ or else $S_t$ would contain $P'$, since $P$ is finite and has no consecutive columns or rows with distance less than $c$. Thus each subset $[S_t]_{u, v}$ has $\mu_2([S_t]_{u, v}) \le \px(\lceil \frac{\lceil \frac{n}{c} \rceil}{\lceil\frac{q d}{c}\rceil+1} \rceil c, P)$, so $\mu_2(S_t) \le (\lceil\frac{q d}{c}\rceil+1)^2 \px(\lceil \frac{\lceil \frac{n}{c} \rceil}{\lceil\frac{q d}{c}\rceil+1} \rceil c, P)$ for all $t \in \mathbb{N}$. Thus $\px(n, P') = \lim_{t \rightarrow \infty} \mu_2(S_t) \le  (\lceil\frac{q d}{c}\rceil+1)^2 \px(\lceil \frac{\lceil \frac{n}{c} \rceil}{\lceil\frac{q d}{c}\rceil+1} \rceil c, P)$.
\end{proof}

\section{A connection to $\ex(n, M)$}\label{mainresult}

In this section, we show that $\px(n, P)$ encompasses the extremal function $\ex(n, M)$ up to a constant factor. We establish the theorem below by proving an upper bound and a lower bound.

\begin{thm}\label{mainth}
For all finite subsets $P \subseteq \mathbb{R}^2$, $\px(n, P) = \Theta(\ex(n, M_P))$, where the constants in the bound depend on the distances between the rows and columns in $P$.
\end{thm}

Then we use this theorem to derive several corollaries about $\px(n, P)$ from known results about $\ex(n, M)$.

\subsection{$\px(n, P) = \Theta(\ex(n, M_P))$ for all finite subsets $P \subseteq \mathbb{R}^2$}\label{mainsub}

In order to prove Theorem \ref{mainth}, we first provide a construction which implies that $\px(n, P) = \Omega(\ex(n, M_P))$. Then we show that $\px(n, P) = O(\ex(n, M_P))$ by turning $P$-free subsets into 0-1 matrices and applying Lemma \ref{tardos_blank}.

\begin{thm}\label{lowerth}
For all finite subsets $P \subseteq \mathbb{R}^2$, $\px(n, P) = \Omega(\ex(n, M_P))$, where the constant in the bound depends on the distances between the rows and columns in $P$.
\end{thm}

\begin{proof}
Let $n$ be a positive integer, let $d$ be the minimum distance between any two consecutive rows of $P$, let $d'$ be the minimum distance between any two consecutive columns of $P$, and define $c' = \min(d, d')$.
Suppose that $c$ is the maximum positive real number that is at most $\min(c', 1)$ for which $\frac{n}{c}$ is an integer. Note that $c \ge \frac{\min(c', 1)}{2}$, since $\frac{n}{x}$ is an integer whenever $x = \frac{n}{2^j}$ for some $j \in \mathbb{N}$. Let $S$ be a subset of $[0,n]^2$ obtained by starting with any $M_P$-free $\frac{n}{c} \times \frac{n}{c}$ 0-1 matrix $A$ with the maximum possible number of ones, and for each pair of integers $1 \leq i, j \leq \frac{n}{c}$ such that $a_{i, j} = 1$, adding the points in the open square $(c(j-1),c j) \times (n-c i, n - c(i-1))$ to $S$. Note that $S$ is open, since it is a union of open sets.

Suppose for contradiction that $S$ contains $P$. Each point in the copy of $P$ in $S$ must be in a distinct open square of $S$, since the open squares have sidelength $c$, and $c$ is at most both the minimum distance between any two consecutive rows of $P$ and minimum distance between any two consecutive columns of $P$. For any points $p, p'$ in the copy of $P$ in the same row, the open squares that contain $p$ and $p'$ must have the same set of $y$-coordinates. For any points $p, p'$ in the copy of $P$ in different rows, the open squares $U$ and $U'$ that contain $p$ and $p'$ must have disjoint sets of $y$-coordinates, since $c$ is at most the minimum distance between any two consecutive rows of $P$ and both $U$ and $U'$ are open $c \times c$ squares. Similarly, for any points $p, p'$ in the copy of $P$ in the same column, the open squares that contain $p$ and $p'$ must have the same set of $x$-coordinates. For any points $p, p'$ in the copy of $P$ in different columns, the open squares $U$ and $U'$ that contain $p$ and $p'$ must have disjoint sets of $x$-coordinates, since $c$ is at most the minimum distance between any two consecutive columns of $P$ and both $U$ and $U'$ are open $c \times c$ squares. 

For each point in the copy of $P$ in $S$ with coordinates $(x, y)$, we have $a_{\frac{n}{c}+1-\lceil \frac{y}{c} \rceil, \lceil \frac{x}{c} \rceil}= 1$. Note that for any points $p, p'$ in the copy of $P$ that are in the open squares $U, U'$, the ones in $A$ corresponding to $p$ and $p'$ must be in the same row if $U$ and $U'$ have the same set of $y$-coordinates, so the ones in $A$ corresponding to $p$ and $p'$ must be in the same row if $p$ and $p'$ are in the same row. The ones in $A$ corresponding to $p$ and $p'$ must be in different rows if $U$ and $U'$ have disjoint sets of $y$-coordinates, so the ones in $A$ corresponding to $p$ and $p'$ must be in different rows if $p$ and $p'$ are in different rows. Similarly, the ones in $A$ corresponding to $p$ and $p'$ must be in the same column if $U$ and $U'$ have the same set of $x$-coordinates, so the ones in $A$ corresponding to $p$ and $p'$ must be in the same column if $p$ and $p'$ are in the same column. The ones in $A$ corresponding to $p$ and $p'$ must be in different columns if $U$ and $U'$ have disjoint sets of $x$-coordinates, so the ones in $A$ corresponding to $p$ and $p'$ must be in different columns if $p$ and $p'$ are in different columns.

Thus the ones in $A$ form a copy of $M_P$ in $A$, a contradiction of $A$ being $M_P$-free. Thus $S$ is $P$-free, and $\mu_2(S) = c^2 \ex(\frac{n}{c}, M_P) \ge c^2 \ex(n, M_P) = \Omega(\ex(n, M_P))$, where the first inequality follows by super-additivity of the extremal function $\ex(n, M_P)$.
\end{proof}

Next we complete the proof of Theorem \ref{mainth} to show that the construction in Theorem \ref{lowerth} is sharp up to a constant factor.

\begin{thm}
For all finite subsets $P \subseteq \mathbb{R}^2$, $\px(n, P) = O(\ex(n, M_P))$, where the constant in the bound depends on the distances between the rows and columns in $P$.
\end{thm}

\begin{proof}
Let $n$ be a positive integer. If $P$ only has a single point, then $\px(n, P) = \ex(n, M_P) = 0$. Thus, we may suppose that $P$ has multiple points, so $M_P$ has multiple ones. It is a well-known fact that any 0-1 matrix $M$ with multiple ones has $\ex(n, M) \ge n$, so we can conclude that $\ex(n, M_P) \ge n$ since $P$ has multiple points.

By Lemma \ref{ps_blank}, it suffices to prove this result for subsets $P$ where the distance between consecutive rows and consecutive columns is equal to $1$. Let $S_t \subseteq [0, n]^2$ for $t \in \mathbb{N}$ be a sequence of $P$-free open sets with the property that $\lim_{t \rightarrow \infty} \mu_2(S_t) = \px(n, P)$. Since each $S_t$ is open, we can define an infinite family of unions of open squares $Z_{t, r} \subseteq S_t$ for $r \in \mathbb{Z}^{+}$ with Lebesgue measures converging to $\mu_2(S_t)$. 

Specifically let $Z_{t, r}$ be obtained from $S_t$ by drawing an $r \times r$ grid of squares each of dimensions $\frac{n}{r} \times \frac{n}{r}$ on $[0, n]^2$, and including the whole interior $U$ of each square in $Z_{t, r}$ if and only if $U \subseteq S_t$, and otherwise including no part of $U$ in $S_t$. We do not include any points in $Z_{t, r}$ from interiors $U$ of squares for which $U \not \subseteq S_t$. We also do not include any points at the boundaries of the $\frac{n}{r} \times \frac{n}{r}$ squares in $Z_{t, r}$. 

For each point $p \in S_t$, let $w_p$ be the maximum positive real number of the form $\frac{n}{x}$ for any positive integer $x$ such that the open ball $B(p, w_p)$ is a subset of $S_t$. Since $S_t$ is open, $w_p$ is defined for all $p \in S_t$. Inside each open ball $B(p, w_p)$, we can draw an open square $T_p$ of sidelength $w_p$ with axis-parallel sides centered at $p$. Note that $S_t = \cup_{p \in S_t} T_p$. 

Fix a positive integer $r$. We can classify the points in $S_t$ that will not be included in $Z_{t, r}$ into two sets $B_{t,r}$ and $L_{t,r}$. If $p = (x, y) \in S_t$ satisfies $x = i \frac{n}{r}$ or $y = j \frac{n}{r}$ for some integers $i$ and $j$, then $p$ is on the boundary of a $\frac{n}{r} \times \frac{n}{r}$ square, so $p$ is not included in $Z_{t, r}$. Let $B_{t,r}$ be the set of points in $S_t$ that are on the boundary of a $\frac{n}{r} \times \frac{n}{r}$ square. If $p \in S_t$ is not in $B_{t,r}$, then the only way that $p$ will not be included in $Z_{t, r}$ is if the open square $U$ of sidelength $\frac{n}{r}$ which contains $p$ is not a subset of $S_t$. If $w_p \ge 2\frac{n}{r}$, then $U \subseteq T_p \subseteq S_t$. Thus if $U$ is not a subset of $S_t$, then we must have $w_p < 2 \frac{n}{r}$. Let $L_{t,r}$ be the set of points $p$ in $S_t$ that are not in $B_{t,r}$ and are not in $Z_{t, r}$. Note that $S_t = B_{t,r} \cup L_{t,r} \cup Z_{t, r}$, and the sets $B_{t,r}$, $L_{t,r}$, and $Z_{t, r}$ are disjoint.

Thus $\mu_2(S_t) =  \mu_2(B_{t,r} \cup L_{t,r} \cup Z_{t, r}) = \mu_2(B_{t,r})+\mu_2(L_{t,r})+\mu_2(Z_{t, r}) = \mu_2(L_{t,r})+\mu_2(Z_{t, r})$. Since $w_p < 2 \frac{n}{r}$ for all $p \in L_{t, r}$, we have  $\lim_{r \rightarrow \infty} \mu_2(L_{t,r}) = 0$. Therefore $\lim_{r \rightarrow \infty} \mu_2(Z_{t, r}) = \mu_2(S_t)$, so we have $\lim_{t \rightarrow \infty}\lim_{r \rightarrow \infty} \mu_2(Z_{t, r}) = \px(n, P)$.

For each subset $Z_{t, r}$, observe that we can define an $r \times r$ 0-1 matrix $A_{t, r}$ for which the $(i, j)$-entry equals $1$ if and only if the square in row $i$ and column $j$ of the grid is included in $Z_{t, r}$. Our numbering of the squares in the grid goes from left to right and top to bottom, like a matrix. Note that $Z_{t, r}$ avoids $P$, since $S_t$ avoids $P$ and $Z_{t, r} \subseteq S_t$. Thus $A_{t, r}$ avoids $S(M_P, \lceil \frac{r}{n} \rceil)$, or else the ones in $S(M_P, \lceil \frac{r}{n} \rceil)$ would correspond to a union of interiors of squares in $Z_{t, r}$ that contain a copy of $P$. 

Let $|A_{t, r}|$ denote the number of ones in $A_{t, r}$. For each $t \geq 0$ and $r \geq n^2$, we have $\mu_2(Z_{t, r}) = (\frac{n}{r})^2 |A_{t, r}| \le (\frac{n}{r})^2 \ex(r, S(M_P, \lceil \frac{r}{n} \rceil)) \le  (\frac{n}{r})^2 (\lceil \frac{r}{n}\rceil+1)^2 \ex(\lceil \frac{r}{\lceil \frac{r}{n} \rceil+1 } \rceil, M_P) = (1+o(1)) \ex(\lceil \frac{r}{\lceil \frac{r}{n} \rceil+1 } \rceil, M_P) \le (1+o(1)) \ex(n, M_P)$, where the first inequality follows from the preceding paragraph and the second inequality follows from Lemma \ref{tardos_blank}. Thus $\px(n, P) = \lim_{t \rightarrow \infty}\lim_{r \rightarrow \infty} \mu_2(Z_{t, r}) = O(\ex(n, M_P))$.
\end{proof}

\subsection{Corollaries of results on $\ex(n, M)$}\label{01cor}

Marcus and Tardos proved that $\ex(n, M) = O(n)$ for every permutation matrix $M$ \cite{MT}, which was the missing piece to prove the Stanley-Wilf conjecture.   Fox later sharpened the bound by proving that $\ex(n, M) = 2^{O(k)}n$ for every $k \times k$ permutation matrix $M$. This implies the corollary  below.

\begin{cor}
Let $P \subseteq \mathbb{R}^2$ with $|P| = k$ for which all points in $P$ have different $x$-coordinates and different $y$-coordinates, and the distances between all consecutive rows and all consecutive columns are at most $c$. Then $\px(n, P) = 2^{O(k)}n$, where the constant in the bound depends on $c$. 
\end{cor}

The K\H{o}vari-S\'{o}s-Tur\'{a}n theorem states that $\ex(n, J_{s,t}) = O(n^{2-\frac{1}{t}})$, where $J_{s, t}$ denotes the $s \times t$ matrix of all ones. As a result, for every 0-1 matrix $M$ there exists a constant $\epsilon > 0$ such that $\ex(n, M) = O(n^{2-\epsilon})$. This is because $M$ is contained in the all-ones matrix of the same dimensions, so we can let $r$ be the width of $M$ and set $\epsilon = \frac{1}{r}$.

\begin{cor}
For every subset $P$ with $s$ rows, $t$ columns, and $s t$ points, we have $\px(n, P) = O(n^{2-\frac{1}{t}})$.
\end{cor}

We prove a stronger version of this result in Section \ref{infinite_p} for infinite sets of points that look like equal signs ($=$) when $t = 2$ and equivalence symbols ($\equiv$) when $t = 3$.

\begin{cor}
For every finite subset $P \subseteq \mathbb{R}^2$ there exists a constant $\epsilon > 0$ such that $\px(n, P) = O(n^{2-\epsilon})$.
\end{cor}

We show that the last corollary is false for bounded infinite subsets in Section \ref{infinite_p}. The K\H{o}vari-S\'{o}s-Tur\'{a}n theorem is known to be sharp for $s \ge 2$ and $t = 2$, as well as $s \ge 3$ and $t = 3$. Thus, we obtain the following corollary about $\px(n, P)$. 

\begin{cor}\label{kst2p}
For every subset $P \subseteq \mathbb{R}^2$ with $k \ge 2$ rows, $2$ columns, and $2 k$ points, we have $\px(n, P) = \Theta(n^{3/2})$. For every subset $P \subseteq \mathbb{R}^2$ with $k \ge 3$ rows, $3$ columns, and $3 k$ points, we have $\px(n, P) = \Theta(n^{5/3})$.
\end{cor}

Keszegh \cite{Ke} showed that any 0-1 matrix $M$ that has no rows with multiple ones satisfies $\ex(n, M) = O(n 2^{\alpha(n)^t})$ for some constant $t$ that depends on $M$, where $\alpha(n)$ denotes the extremely slow-growing inverse Ackermann function. This implies a corresponding result for finite subsets $P \subseteq \mathbb{R}^2$.

\begin{cor}
If $P$ is a finite subset of $\mathbb{R}^2$ that has no rows with multiple points, then $\px(n, P) = O(n 2^{\alpha(n)^t})$ for some constant $t$ that depends on $P$. 
\end{cor}

Call a 0-1 matrix $M$ \emph{linear} if $\ex(n, M) = O(n)$ and \emph{nonlinear} otherwise. Similarly we call a subset $P \subseteq \mathbb{R}^2$ linear if $\px(n, P) = O(n)$ and nonlinear otherwise. Several papers on $\ex(n, M)$ have identified families of linear 0-1 matrices (e.g. permutation matrices, double permutation matrices, matrices corresponding to visibility graphs, and linear matrices with at most four ones in \cite{fulek, gduluth, gshen, MT, tardos05}), and each of these linear 0-1 matrices corresponds to linear subsets of $\mathbb{R}^2$. Other papers have identified families of nonlinear 0-1 matrices (e.g. block permutation matrices, minimally nonlinear matrices, nonlinear matrices with applications to path minimization algorithms, and nonlinear matrices with at most four ones in \cite{hesterberg, Ke, mitchell, pettie, tardos05}), and each of these nonlinear 0-1 matrices corresponds to nonlinear subsets of $\mathbb{R}^2$.

\section{Forbidden bounded infinite subsets}\label{infinite_p}

On the topic of linear subsets, we start with one of the most natural infinite forbidden subsets to consider, a line segment $P$. 

\begin{thm}\label{horizseg}
If $P$ is a horizontal line segment of length $c$ with closed endpoints, then $\px(n, P) = c n$.
\end{thm}

\begin{proof}
Suppose that $S \subseteq [0, n]^2$ is open and $P$-free. For each $y \in [0, n]$, let $S_y$ denote the set of points $(a, b) \in S$ such that $b = y$. By Lemma \ref{usefullem}, we have  $\mu_1(S_y) \le c$ for all $y \in [0, n]$.

Since $S$ is open, the function $\mu_1(S_y): [0, n] \rightarrow \mathbb{R}$ is a measurable function by Fubini's theorem, so it is Lebesgue integrable on $[0, n]$. Thus $\mu_2(S) = \int_{y \in [0, n]} \mu_1(S_y) d\mu_1(y) \le c n$, so $\px(n, P) \le cn$. On the other hand, the subset $T = (0, c) \times (0, n)$ is open and $P$-free with $\mu_2(T) =  c n$, so $\px(n, P) \ge \mu_2(T) = c n$.
\end{proof}

Note that by symmetry, the last result also applies to vertical line segments. Also, the lower bound construction in the last result shows that Corollary \ref{addedseg} and Corollary \ref{addedpt} are sharp when applied to a single point. Next we consider line segments that are neither horizontal nor vertical.

\begin{thm}\label{diagseg}
If $P$ is a line segment with closed endpoints between $(0, 0)$ and $(a, b)$ for some $a, b > 0$, then $\px(n, P) = (a+b)n - a b$.
\end{thm}

\begin{proof}
Suppose that $S \subseteq [0, n]^2$ is open and $P$-free. Rotate $S$, $[0, n]^2$, and $P$ clockwise all by the same angle, until $P$ becomes a horizontal segment $R(P)$ of length $\sqrt{a^2+b^2}$ with endpoints at $(0, 0)$ and $(\sqrt{a^2+b^2}, 0)$. The result of rotating $[0, n]^2$ is denoted $R([0,n]^2)$, and the result of rotating $S$ is denoted $R(S)$. The rows of $R(S)$ are a subset of the rows of $R([0,n]^2)$, which has height $\frac{(a+b)n}{\sqrt{a^2+b^2}}$. Suppose that we also translate $R([0, n]^2)$ and $R(S)$ so that all points in both sets lie between $y = 0$ and $y = \frac{(a+b)n}{\sqrt{a^2+b^2}}$. For each $y \in [0, \frac{(a+b)n}{\sqrt{a^2+b^2}}]$, let $R(S)_y$ denote the set of points $(u, v) \in R(S)$ with $v = y$, and let $R([0,n]^2)_y$ denote the set of points $(u, v) \in R([0,n]^2)$ with $v = y$. First, we claim that $R(S)$ must avoid $R(P)$.

Suppose for contradiction that $R(S)$ contains $R(P)$. Then there exists $y$ such that $R(S)_y$ contains $R(P)$. Let $L$ be the set of $x$-coordinates in $R(S)_y$, so there is an injection $f: [0, \sqrt{a^2+b^2}] \rightarrow L$ such that $f(u) - f(v) \ge u-v$ for all $u, v \in [0, \sqrt{a^2+b^2}]$ with $u > v$. Let $Q = R^{-1}(R(S)_y)\subseteq S$ be the preimage under $R$ of $R(S)_y$. In particular let $(d, e)$ be the preimage under $R$ of $(f(0),y)$. Let $X_Q$ denote the set of $x$-coordinates of the points in $Q$, and let $Y_Q$ denote the set of $y$-coordinates of the points in $Q$.

Define $g_X: [0, a] \rightarrow X_Q$ and $g_Y: [0, b] \rightarrow Y_Q$ by $g_X(x) = d+\frac{a}{\sqrt{a^2+b^2}} (f(\frac{\sqrt{a^2+b^2}}{a}x)-f(0))$ and $g_Y(y) = e+\frac{b}{\sqrt{a^2+b^2}} (f(\frac{\sqrt{a^2+b^2}}{b}y)-f(0))$. For any $x_1, x_2 \in [0, a]$ with $x_1 > x_2$, we have $g_X(x_1) - g_X(x_2) =  \frac{a}{\sqrt{a^2+b^2}}(f(\frac{\sqrt{a^2+b^2}}{a}x_1)-f(\frac{\sqrt{a^2+b^2}}{a}x_2)) \ge \frac{a}{\sqrt{a^2+b^2}}(\frac{\sqrt{a^2+b^2}}{a}x_1 -\frac{\sqrt{a^2+b^2}}{a}x_2) = x_1-x_2$. Similarly $g_Y(y_1)-g_Y(y_2) \ge y_1-y_2$ for all $y_1, y_2 \in [0, b]$ with $y_1 > y_2$. Moreover for any point $(a t, b t)$ with $t \in [0, 1]$, we can see that $(g_X(a t), g_Y(b t)) \in Q$. This is because $f(0) \in L$, $f(t\sqrt{a^2+b^2}) \in L$ and $R^{-1}$ maps $(f(0), y)$ to $(d, e)$, so $R^{-1}$ maps $(f(t\sqrt{a^2+b^2}),y)$ to $(d+\frac{a}{\sqrt{a^2+b^2}}(f(t\sqrt{a^2+b^2})-f(0)), e+\frac{b}{\sqrt{a^2+b^2}}(f(t\sqrt{a^2+b^2})-f(0))) = (g_X(a t), g_Y(b t))$. Since $(f(t\sqrt{a^2+b^2}),y) \in R(S)_y$ and $Q = R^{-1}(R(S)_y)$, we must have $(g_X(a t), g_Y(b t)) \in Q$. So $Q$ contains $P$, implying $S$ contains $P$, which is a contradiction. Thus $R(S)$ must avoid $R(P)$.

Then $\mu_1(R(S)_y) \le \sqrt{a^2+b^2}$ for all $y \in [0,\frac{(a+b)n}{\sqrt{a^2+b^2}}]$, or else $R(S)$ would contain $R(P)$ by Lemma \ref{usefullem} with $c = \sqrt{a^2+b^2}$. 

However for all $y$ such that $\frac{(a+b)n}{\sqrt{a^2+b^2}} - \frac{a b}{\sqrt{a^2+b^2}} \le y \le  \frac{(a+b)n}{\sqrt{a^2+b^2}}$, we have $\mu_1(R(S)_y) \le \mu_1(R([0,n]^2)_y) = (\frac{(a+b)n}{\sqrt{a^2+b^2}} - y) \frac{a^2+b^2}{a b}$. Similarly if $y \le  \frac{a b}{\sqrt{a^2+b^2}}$, then $\mu_1(R(S)_y) \le \mu_1(R([0,n]^2)_y) = y \frac{a^2+b^2}{a b}$. Since $R(S)$ is open, the function $\mu_1(R(S)_y): [0, \frac{(a+b)n}{\sqrt{a^2+b^2}}] \rightarrow \mathbb{R}$ is a measurable function by Fubini's theorem, so it is Lebesgue integrable on $[0, \frac{(a+b)n}{\sqrt{a^2+b^2}}]$. Thus $\mu_2(S) = \mu_2(R(S)) = \int_{y \in [0, \frac{(a+b)n}{\sqrt{a^2+b^2}}]} \mu_1(R(S)_y) d\mu_1(y) \le (\frac{(a+b)n}{\sqrt{a^2+b^2}} - \frac{2a b}{\sqrt{a^2+b^2}})\sqrt{a^2+b^2}+\frac{a b}{\sqrt{a^2+b^2}}\sqrt{a^2+b^2} = (a+b)n - a b$.

On the other hand, the subset $T = \left\{(0, a) \times (0, n)\right\} \cup \left\{(0, n) \times (0, b)\right\}$ is open and $P$-free with $\mu_2(T) =  a n + b n - a b$, so $\px(n, P) \ge  \mu_2(T) = (a+b)n - a b$.
\end{proof}

We can use the last result to get a linear bound on $\px(n, P)$ for a much more general family of subsets $P$.

\begin{thm}\label{autoarc}
Suppose that $f: [0, a] \rightarrow \mathbb{R}$ is increasing, with $\frac{f(t)-f(s)}{t-s} \le b$ for all $s, t \in [0, a]$ with $s < t$. If $P$ is the set of points $\left\{(t,f(t)): t \in [0, a]\right\}$, then $\px(n, P) = \Theta(n)$.
\end{thm}

\begin{proof}
Let $Q$ be the set of points $\left\{(t, b t) \in \mathbb{R}^2: t \in [0, a] \right\}$. Then $Q$ contains $P$, as evidenced by the maps $g_X: [0, a] \rightarrow [0, a]$ and $g_Y: f([0,a]) \rightarrow [0, b a]$ defined by $g_X(t) = t$ and $g_Y(f(t)) = b t$. So $\px(n, P) \le \px(n, Q) = O(n)$ by Theorem \ref{diagseg} and Lemma \ref{pxcontain}. On the other hand, we have $\px(n, P) \ge a n$ since $(0, a) \times (0, n)$ avoids $P$, so $\px(n, P) = \Theta(n)$.
\end{proof}

\subsection{Bounded infinite $P \subseteq \mathbb{R}^2$ with $\px(n, P) = \Theta(n^2)$}\label{bin2}

In the last few results, we saw infinite subsets $P$ with $\px(n, P) = O(n)$. In Section \ref{mainresult} we showed that $\px(n, P) = \Theta(\ex(n, M_P))$ for all finite subsets $P$, and it follows from the K\H{o}vari-S\'{o}s-Tur\'{a}n theorem that for every 0-1 matrix $M$ there exists $\epsilon > 0$ such that $\ex(n, M) = O(n^{2-\epsilon})$, so for every finite subset $P$  there exists $\epsilon > 0$ such that $\px(n, P) = O(n^{2-\epsilon})$. In the next result, we show that there exist bounded and countably infinite sets of points $P$ for which $\px(n, P) = \Omega(n^{2})$. For the next proof, we use the well-known fact that $\ex(n, J_{r, r}) = \Omega(n^{2-\frac{2}{r+1}})$ for all $r \ge 2$, where the constant in the bound does not depend on $r$. 

This fact is quick to prove using probabilistic methods as in \cite{erdoskrr}. Suppose that we choose a random $n \times n$ 0-1 matrix where each entry is $1$ with probability $p$. The expected number of copies of $J_{r, r}$ is $p^{r^2} \binom{n}{r} \binom{n}{r}$. We can delete a one from each copy, so the expected number of ones in the altered 0-1 matrix is at least $p n^2 - p^{r^2} \binom{n}{r} \binom{n}{r} > \frac{p n^2}{2}$ when $r \ge 2$ and $p \le n^{-\frac{2}{r+1}}$. Thus $\ex(n, J_{r, r}) \ge \frac{1}{2} n^{2-\frac{2}{r+1}}$ for all $r \ge 2$.

\begin{thm}
If $P \subseteq \mathbb{R}^2$ is open, then $\px(n, P) = \Theta(n^2)$, where the constant in the lower bound depends on $P$.
\end{thm}

In order to prove this theorem, we will prove a stronger fact. Given an open subset $P \subseteq \mathbb{R}^2$, let $Q_P$ be the set of points in $P$ with rational coordinates. In the following proof, all logarithms are base 2.

\begin{thm}\label{open2}
If $P \subseteq \mathbb{R}^2$ is open, then $\px(n, Q_P) = \Theta(n^2)$, where the constant in the lower bound depends on $P$.
\end{thm}

\begin{proof}
The upper bound is trivial. For the lower bound, it suffices to show that there exists a subset $Q \subseteq Q_P$ for which $\px(n, Q) =  \Omega(n^{2})$, where the constant in the bound depends on $P$. Let $p$ be an arbitrary element of $P$. Since $P$ is open, there exists some $s > 0$ for which there is an open ball $B(p, s)$ centered at $p$ of radius $s$ with $B(p, s) \subseteq P$. Since $s > 0$, there exists some rational number $0 < c < 1$ such that $B(p, s)$ contains a closed square of sidelength $c$ with vertices at rational coordinates. Without loss of generality, we may assume that this square is $[0, c]^2$. 

Let $Q_r = \left\{\frac{i c}{r}: i = 1, 2, \dots, r \right\}^2$ for $r \ge 2$, so that $|Q_r| = r^2$. If $H_r = \left\{1, 2, \dots, r \right\}^2$ for $r \ge 2$, then Lemma \ref{ps_blank} implies that $\px(m, H_r) \le (\lceil \frac{r}{c} \rceil+1)^2 \px(\lceil \frac{\lceil \frac{m r}{c} \rceil}{\lceil \frac{r}{c} \rceil+1} \rceil \frac{c}{r}, Q_r)$. By Theorem \ref{mainth} and the cited lower bound on $\ex(n, J_{r, r})$ in the paragraph before this proof, $\px(m, H_r) = \Omega(m^{2-\frac{2}{r+1}})$ for all $r \ge 2$, where the constant in the bound does not depend on $r$. Thus $\px(\lceil \frac{\lceil \frac{m r}{c} \rceil}{\lceil \frac{r}{c} \rceil+1} \rceil \frac{c}{r}, Q_r) =  \Omega(\frac{m^{2-\frac{2}{r+1}}}{(\lceil \frac{r}{c} \rceil+1)^2})$ for all $r \ge 2$. If $m \ge 4$ is an integer and $r = \lceil \log{m} \rceil$, we obtain $\px(\frac{m}{\log{m}}, Q_{r}) \ge \px(\frac{m c}{r}, Q_{r}) \ge \px(\lceil \frac{\lceil \frac{m r}{c} \rceil}{\lceil \frac{r}{c} \rceil+1} \rceil \frac{c}{r}, Q_{r}) =  \Omega(\frac{m^{2-\frac{2}{r+1}}}{(\lceil \frac{r}{c} \rceil+1)^2}) = \Omega(\frac{m^2}{(\log{m})^2})$. Let $n$ satisfy $m = n \log{n}$ for some integer $m \ge 4$. Thus we obtain $\px(n, Q_P) \ge \px(\frac{n \log{n}}{\log(n \log{n})}, Q_P) \ge  \px(\frac{n \log{n}}{\log(n \log{n})}, Q_{\lceil \log(n \log{n}) \rceil}) =  \Omega(n^2)$.
\end{proof}

\subsection{Strengthening the K\H{o}vari-S\'{o}s-Tur\'{a}n theorem}\label{skst}

For each $t \ge 2$, let $P_{s, t, c}$ denote the set of points $\left\{(x, y): x \in [0, s] \text{ and } y \in \left\{c, 2c, \dots, t c\right\}  \right\}$. For example, $P_{s, 2, c}$ is a set of points that looks like an equal sign ($=$) and $P_{s, 3, c}$ is a set of points that looks like an equivalence symbol ($\equiv$).

We start by determining $\px(n, P_{s, 2, c})$ up to a constant factor that depends on $c$ before we prove a general upper bound on $\px(n, P_{s, t, c})$. The integrals in the next two proofs are Lebesgue integrals. For the lower bound in the next proof, we use the result of F\"{u}redi \cite{fur2} that $\ex(n, J_{s, 2}) = \Theta(s^{\frac{1}{2}}n^{\frac{3}{2}})$.

\begin{thm}\label{t=2}
For all $s > 0$, $\px(n, P_{s, 2, c}) = \Theta(s^{\frac{1}{2}}n^{\frac{3}{2}})$.
\end{thm}

\begin{proof}
The lower bound follows from Lemma \ref{pxcontain}, Theorem \ref{mainth}, and the result of F\"{u}redi cited in the paragraph before this proof. For the upper bound, let $S$ be an open $P_{s, 2, c}$-free subset of $[0,n]^2$. Let $S'$ be the $3$-dimensional set of points of the form $(x, y, z)$ for which $(x, y) \in S$ and $(x, z) \in S$ and $y - z > c$. First, we note that $S'$ is open. To see why this is true, define $T_1 = \left\{(x, y, z) \in \mathbb{R}^3: (x, y) \in S \text{ and } (x, z) \in S \right\}$ and define $T_2 = \left\{(x, y, z) \in \mathbb{R}^3: y - z > c \right\}$. Clearly $T_2$ is an open subset of $\mathbb{R}^3$. Since $S$ is open, for every point $(x, y) \in S$ there exists $r > 0$ such that the open ball of radius $r$ centered at $(x, y)$ is a subset of $S$, i.e. $B((x, y), r) \subseteq S$. For each $(x, y, z) \in T_1$, we have both $(x, y) \in S$ and $(x, z) \in S$, so there exist $r_y, r_z > 0$ such that $B((x, y), r_y) \subseteq S$ and $B((x, z), r_z) \subseteq S$. For every $(x',y',z') \in B((x, y, z), \min(r_y, r_z))$, we have $(x'-x)^2+(y'-y)^2+(z'-z)^2 < \min(r_y^2, r_z^2)$, so $(x'-x)^2+(y'-y)^2 < r_y^2$ and $(x'-x)^2+(z'-z)^2 < r_z^2$, which implies that $(x', y') \in S$ and $(x',z') \in S$, so $(x',y',z') \in T_1$. Thus $B((x, y, z), \min(r_y, r_z)) \subseteq T_1$, so $T_1$ is open. Hence $S'$ is also open, since $S' = T_1 \cap T_2$.

The points $(x, y, z) \in S'$ must satisfy $0 < z < n-c$ and $z+c < y < n$, so the projection of these points onto the $y-z$ plane is the interior of a triangle of area $\frac{(n-c)^2}{2}$. For each fixed $y$ and $z$ with $y, z \in (0, n)$, let $S''_{y, z}$ denote the set of points $(a, b, d) \in S'$ with $b = y$ and $d = z$. Then $\mu_1(S''_{y z}) \le s$ by Lemma \ref{usefullem} for all  $y, z \in (0, n)$, or else $S$ would contain $P_{s, 2, c}$. Since $S'$ is open, by Fubini's theorem the function $\mu_1(S''_{y z}): [0, n]^2 \rightarrow \mathbb{R}$ is a measurable function, so it is Lebesgue integrable on $[0, n]^2$. Thus $\mu_3(S') = \int_{z \in (0, n)} \int_{y \in (0, n)} \mu_1(S''_{y z}) d\mu_1(y) d\mu_1(z)  \le \frac{s(n-c)^2}{2}$.

For all $x \in [0, n]$, let $S'_x$ denote the set of points $(a, b, d) \in S'$ with $a = x$, let $S_x$ denote the set of points $(a, b) \in S$ with $a = x$, and let $m_x = \mu_1(S_x)$. Since $S$ is an open set, by Fubini's theorem the function $m_x: [0, n] \rightarrow \mathbb{R}$ is a measurable function. Therefore $\frac{\max(m_x-c,0)^2}{2}$ is also a measurable function. Thus $\frac{\max(m_x-c,0)^2}{2}$ is Lebesgue integrable on $[0, n]$. 

For all $x$ with $m_x \ge c$, we must have $\mu_2(S'_x) \ge \frac{(m_x-c)^2}{2}$. To see why this is true, suppose that $m_x \ge c$. For each $(x, y) \in S_x$, let $Q_{x, y} = \left\{(x, r): (x, r) \in S \text{ and } r \le y \right\}$. Note that 

\begin{align*}
\left\{(y, z): (x, y) \in S \text{ and } (x, z) \in S \text{ and } \mu_1(Q_{x, z})+ c < \mu_1(Q_{x,y}) \right\}  \subseteq \\
\left\{(y, z): (x, y) \in S \text{ and } (x, z) \in S \text{ and } z+c < y  \right\}.
\end{align*}

\noindent Then we have $\mu_2(S'_x) =$\\ 

\begin{align*}
\int_{z \in (0, n-c)} \int_{y \in (z+c, n)} 1_{(x,z) \in S} 1_{(x,y) \in S} d \mu_1(y) d \mu_1(z) \ge \\
\int_{z \in (0, n-c)} \int_{y \in (z+c, n)} 1_{(x,z) \in S} 1_{(x,y) \in S} 1_{\mu_1(Q_{x, z})+ c < \mu_1(Q_{x,y})} d \mu_1(y) d \mu_1(z). 
\end{align*}

Given any $(x, y) \in S$, note that $ \mu_1(Q_{x, y}) \in (0, m_x)$. Since $S$ is open, we can write $S_x$ as a countable union of intervals $I_j$ with open endpoints. Thus there are only countably many $b \in (0, m_x)$ for which there does not exist $r$ such that $\mu_1(Q_{x,r}) = b$. Since countable sets have measure zero, the last integral is equal to 

\begin{align*}
\int_{z' \in (0, m_x-c)} \int_{y' \in (z'+c,m_x)} 1 d \mu_1(y') d \mu_1(z') \\
= \frac{(m_x-c)^2}{2}
\end{align*}

\noindent Thus by Fubini's theorem we have $\mu_3(S') = \int_{x \in [0, n]} \mu_2(S'_x) d\mu_1(x) \ge \int_{x \in [0, n]} \frac{\max(m_x-c,0)^2}{2} d\mu_1(x)$.

Let $f: \mathbb{R} \rightarrow \mathbb{R}$ be defined by $f(t) = \frac{\max(t-c,0)^2}{2}$, so $f(m_x) = \frac{\max(m_x-c,0)^2}{2}$ and $f$ is convex. Thus we can rewrite the last inequality in the last paragraph as $\mu_3(S') \ge \int_{x \in [0,n]} f(m_x) d\mu_1(x)$. Combining this with the inequality at the end of the second paragraph, we have $\int_{x \in [0, n]} f(m_x) d\mu_1(x) \le \frac{s(n-c)^2}{2}$. By Jensen's inequality, we obtain $n f(\frac{\mu_2(S)}{n}) \le \frac{s(n-c)^2}{2}$, so $n \frac{(\frac{\mu_2(S)}{n}-c)^2}{2} \le \frac{s(n-c)^2}{2}$ or else $\frac{\mu_2(S)}{n} \le c$. Thus $\mu_2(S) = O(s^{\frac{1}{2}}n^{\frac{3}{2}})$.
\end{proof}

In order to generalize the last theorem, we prove a lemma where we bound the volume of a $t$-dimensional solid that will be used in the main proof.

\begin{lem}\label{simplex}
The set $X$ of points $(y_1, \dots, y_t)$ for which $0 < y_t < n-(t-1)c$ and $y_{i+1}+c < y_i < n-(i-1)c$ for all $i = 1, \dots, t-1$ has $\mu_t(X) = \Theta(n^t)$, where the constants in the bound depend on $t$. In particular, $\mu_t(X) \ge (\frac{n}{t}-c)^t$ for $n \ge c t$.
\end{lem}

\begin{proof}
$X$ is contained in the set of points $(y_1, \dots, y_t)$ with $0 \le y_i \le n$ for all $i$, which has volume $n^t$, so $\mu_t(X) \le n^t$. For a lower bound, suppose that $n > c t$ and define a set of points $X'$ consisting of the points $(y_1, \dots, y_t)$ such that $y_{t+1-i} \in (\frac{n}{t}(i-1),\frac{n}{t}i - c)$ for each $i = 1, \dots, t$. Then $X' \subseteq X$, and $\mu_t(X')  = (\frac{n}{t}-c)^t$. Thus $\mu_t(X) \ge \mu_t(X') = \Omega(n^t)$.
\end{proof}

While the last theorem covered subsets that look like the equal sign ($=$), the next theorem covers subsets that look like equivalence symbols ($\equiv$) and more generally, vertical stacks of any finite number of horizontal segments of the same length with endpoints in the same left and right columns. 

\begin{thm}\label{kstcont}
For all $s > 0$ and fixed $t \ge 2$, $\px(n, P_{s, t, c}) = O(s^{\frac{1}{t}} n^{2-\frac{1}{t}})$, where the constants in the bound depend on $t$ and $c$.
\end{thm}

\begin{proof}
Let $S$ be an open $P_{s, t, c}$-free subset of $[0,n]^2$. Let $S'$ be the $(t+1)$-dimensional set of points of the form $(x, y_1, y_2, \dots y_t)$ for which $(x, y_i) \in S$ for each $i = 1, 2, \dots, t$ and $y_i - y_{i+1} > c$ for each $i = 1, \dots, t-1$. First, we note that $S'$ is open. To see why this is true, as in Theorem \ref{t=2} define $T_1 = \left\{(x, y_1, y_2, \dots y_t) \in \mathbb{R}^{t+1}: (x, y_i) \in S \text{ for all } i = 1, 2, \dots, t \right\}$. Define $T_2 = \left\{(x, y_1, y_2, \dots y_t) \in \mathbb{R}^{t+1}: y_i - y_{i+1} > c \text{ for all } i = 1, 2, \dots, t-1 \right\}$. Clearly $T_2$ is an open subset of $\mathbb{R}^{t+1}$. Since $S$ is open, for every point $(x, y) \in S$ there exists $r > 0$ such that the open ball of radius $r$ centered at $(x, y)$ is a subset of $S$, i.e. $B((x, y), r) \subseteq S$. For each $(x,  y_1, y_2, \dots y_t) \in T_1$, we have $(x, y_i) \in S$ for all $i = 1, 2, \dots, t$, so there exist $r_i > 0$ for each $i = 1, 2, \dots, t$ such that $B((x, y_i), r_i) \subseteq S$ for all $i = 1, 2, \dots, t$. For every $(x', y'_1, y'_2, \dots y'_t) \in B((x, y_1, y_2, \dots y_t), \min(r_1, r_2, \dots, r_t))$, we have $(x'-x)^2+\sum_{i = 1}^t (y'_i-y_i)^2 < \min(r_1^2, r_2^2, \dots, r_t^2)$, so $(x'-x)^2+(y'_i-y_i)^2 < r_i^2$ for all $i = 1, 2, \dots, t$, which implies that $(x', y'_i) \in S$ for all $i = 1, 2, \dots, t$, so $(x', y'_1, y'_2, \dots y'_t) \in T_1$. Thus $B((x, y_1, y_2, \dots y_t), \min(r_1, r_2, \dots, r_t)) \subseteq T_1$, so $T_1$ is open. Hence $S'$ is also open, since $S' = T_1 \cap T_2$.

The points $(x, y_1, y_2, \dots y_t) \in S'$ must satisfy $0 < y_t < n-(t-1)c$ and $y_{i+1}+c < y_i < n-(i-1)c$ for all $i = 1, \dots, t-1$. Thus the projection of these points onto the last $t$ coordinates is a $t$-dimensional solid of volume $\Theta(n^t)$ by Lemma \ref{simplex}. For each fixed $y_1, y_2, \dots, y_t$ with $y_1, y_2, \dots, y_t \in (0, n)$, let $S''_{y_1, y_2, \dots, y_t}$ denote the set of points $(a, b_1, b_2, \dots, b_t) \in S'$ with $b_i = y_i$ for each $i$. Then $\mu_1(S''_{y_1, y_2, \dots, y_t}) \le s$ by Lemma \ref{usefullem} or else $S$ would contain $P_{s, t, c}$. Since $S'$ is open, by Fubini's theorem the function $\mu_1(S''_{y_1, y_2, \dots, y_t}): [0, n]^t \rightarrow \mathbb{R}$ is a measurable function, so it is Lebesgue integrable on $[0, n]^t$. Thus $\mu_{t+1}(S') = \int_{y_t \in (0, n)} \int_{y_{t-1} \in (0,n)} \dots \int_{y_1 \in (0, n)} \mu_1(S''_{y_1, y_2, \dots, y_t}) d\mu_1(y_1) d\mu_1(y_2) \dots d\mu_1(y_t) = O(s n^t)$.

For all $x \in [0, n]$, let $S'_x$ denote the set of points $(a, b_1, b_2, \dots, b_t) \in S'$ with $a = x$, let $S_x$ denote the set of points $(a, b) \in S$ with $a = x$, and let $m_x = \mu_1(S_x)$. As in the last proof, since $S$ is open, by Fubini's theorem the function $m_x: [0, n] \rightarrow \mathbb{R}$ is a measurable function. Therefore $\frac{\max(m_x-c t,0)^t}{t^t}$ is also a measurable function. Thus $\frac{\max(m_x-c t,0)^t}{t^t}$ is Lebesgue integrable over $[0, n]$. For all $x$ with $m_x \ge c t$, we have $\mu_t(S'_x) \ge \frac{\max(m_x-c t,0)^t}{t^t}$.

To see why this is true, suppose that $m_x \ge c t$. For each $(x, y) \in S_x$, as in Theorem \ref{t=2} let $Q_{x, y} = \left\{ (x, r): (x, r) \in S \text{ and } r \le y\right\}$. Note that 

\begin{align*}
\left\{(y_1, \dots, y_t): (\forall i \le t) ((x, y_i) \in S) \text{ and } (\forall i \le t-1) (\mu_1(Q_{x,y_{i+1}})+c < \mu_1(Q_{x,y_i})) \right\}     \subseteq  \\
   \left\{(y_1, \dots, y_t):  (\forall i \le t)  ((x, y_i) \in S) \text{ and }  (\forall i \le t-1) (y_{i+1}+c < y_i) \right\}.
\end{align*} 

\noindent Then we have $\mu_t(S'_x) = $

\footnotesize
\begin{align*}
\int_{y_t \in (0, n-(t-1)c)} \int_{y_{t-1} \in (y_t+c, n-(t-2)c)} \dots \int_{y_1 \in (y_2+c, n)} \prod_{i = 1}^{t} 1_{(x,y_i) \in S} d \mu_1(y_1) d \mu_1(y_2) \dots d \mu_1(y_t) \ge \\
\int_{y_t \in (0, n-(t-1)c)} \int_{y_{t-1} \in (y_t+c, n-(t-2)c)} \dots \int_{y_1 \in (y_2+c, n)} (\prod_{i = 1}^{t}1_{(x,y_i) \in S})(\prod_{i = 1}^{t-1} 1_{\mu_1(Q_{x,y_{i+1}})+c < \mu_1(Q_{x,y_i})}) d \mu_1(y_1) d \mu_1(y_2) \dots d \mu_1(y_t).
\end{align*}
\normalsize

Given any $(y_1, \dots, y_t)$ with $(x, y_i) \in S$ for each $i \le t$, note that we have $\mu_1(Q_{x,y_i}) \in (0, m_x)$ for each $i \le t$. Since $S$ is open, we can write $S_x$ as a countable union of intervals $I_j$ with open endpoints. Thus there are only countably many $b \in (0, m_x)$ for which there does not exist $r$ such that $\mu_1(Q_{x,r}) = b$. Since countable sets have measure zero, the last integral is equal to 

\begin{align*}
\int_{z_t \in (0, m_x-(t-1)c)} \int_{z_{t-1} \in (z_t+c, m_x-(t-2)c)} \dots \int_{z_1 \in (z_2+c, m_x)} 1 d \mu_1(z_1) d \mu_1(z_2) \dots d \mu_1(z_t) \ge \\
 \frac{(m_x-c t)^t}{t^t},
\end{align*}

\noindent where the last inequality follows by Lemma \ref{simplex}. Thus by Fubini's theorem, we obtain $\mu_{t+1}(S') =  \int_{x \in [0, n]}\mu_t(S'_x) d\mu_1(x) \ge \int_{x \in [0, n]} \frac{\max((m_x-c t,0))^t}{t^t} d\mu_1(x)$.

Let $f: \mathbb{R} \rightarrow \mathbb{R}$ be defined by $f(z) =  \frac{\max(z-c t,0)^t}{t^t}$, so $f(m_x) =  \frac{\max(m_x-c t,0)^t}{t^t}$ and $f$ is convex. Thus we can write the last inequality in the last paragraph as $\mu_{t+1}(S') \ge \int_{x \in [0, n]} f(m_x) d\mu_1(x)$. Combining this with the inequality at the end of the second paragraph, we have $\int_{x \in [0, n]} f(m_x) d\mu_1(x) = O(s n^t)$. By Jensen's inequality, we obtain $n f(\frac{\mu_2(S)}{n}) = O(s n^t)$, so $n (\frac{\frac{\mu_2(S)}{n}}{t}-c)^t = O(s n^t)$ or else $\frac{\mu_2(S)}{n} \le c t$. Thus $\mu_2(S) = O(s^{\frac{1}{t}}n^{2-\frac{1}{t}})$.
\end{proof}

\subsection{Disjoint unions of horizontal segments}\label{segperm}

Our next result is for any disjoint union of a finite number of horizontal segments $P$ with no two points having the same $x$-coordinate and no two segments having the same $y$-coordinate, for which we prove a linear upper bound on $\px(n, P)$. This is in analogue with the linear upper bound on $\ex(n, M)$ when $M$ is a double permutation matrix \cite{gduluth}.

\begin{thm}
If $P$ is a disjoint union of a finite number of horizontal segments with no two points having the same $x$-coordinate and no two segments having the same $y$-coordinate, then $\px(n, P) = O(n)$.
\end{thm}

\begin{proof}
By Lemma \ref{pxcontain}, it suffices to prove that $\px(n, P) = O(n)$ with the following restrictions on $P$. We assume that the consecutive rows of $P$ are a distance of $c$ apart, each segment in $P$ has closed endpoints and length $\ell$, and there is a distance of $c$ between the vertical lines that contain the closest endpoints of consecutive segments from left to right. 

Suppose that $S \subseteq [0, n]^2$ is an open $P$-free subset. We draw a grid on $S$ with a total of $\lceil \frac{n}{c} \rceil^2$ squares each of dimensions $c \times c$. The squares in this proof are closed, so they intersect at the boundaries. Note that the rightmost column and topmost row of squares in the grid may cover some area outside of $[0, n]^2$ if $\frac{n}{c}$ is not an integer. We construct a 0-1 matrix $B$ of dimensions $\lceil \frac{n}{c} \rceil \times \lceil \frac{n}{c} \rceil$ so that the $c \times c$ square $U_{i j}$ in the $i^{th}$ row and $j^{th}$ column of squares in the grid corresponds to the $(i, j)$ entry $b_{i j}$ of $B$. 

To construct $B$, we add entries row by row, left to right. Suppose that we are on row $i$ of $B$. If $U_{i j}$ contains no point of $S$, then we set $b_{i j} = 0$. If $U_{i j}$ is the first square in the $i^{th}$ row of squares in the grid to contain any point of $S$, then we set $b_{i j} = 1$. Finally if $U_{i j}$ contains a point of $S$ but some other square in the $i^{th}$ row of squares to the left of $U_{i j}$ contains a point of $S$, then let $j'$ be the greatest integer less than $j$ for which $b_{i j'} = 1$ and let $Q$ denote the restriction of $S$ to the $j - j'+1$ squares between $U_{i j'}$ and $U_{i j}$ inclusive. We set $b_{i j} = 1$ if and only if there exists a row $r$ of $Q$ for which $\mu_1(Q_r) > \ell$, where $Q_r$ denotes the $1$-dimensional restriction of $Q$ to row $r$.

Let $P'$ be the finite subset of $\mathbb{R}^2$ obtained from $P$ by only including the left and right endpoints of each segment, so the cardinality of $P'$ is twice the number of segments in $P$. Note that $M_{P'}$ is a double permutation matrix, so $\ex(n, M_{P'}) = O(n)$ by \cite{gduluth}.

The main observation is that $B$ must avoid $S(M_{P'}, 1)$. To see why this is true, suppose for contradiction that $B$ contains $S(M_{P'}, 1)$, and consider a copy of $S(M_{P'}, 1)$ in $B$. Suppose that there are entries $b_{i j_1}$ and $b_{i j_2}$ in row $i$ of $B$ that contain ones in the copy of $S(M_{P'}, 1)$. Then in the $i^{th}$ row of squares in the grid, if $Q$ is the restriction of $S$ to the squares between $U_{i j_1}$ and $U_{i j_2}$ inclusive, we can find some row $r$ such that $\mu_1(Q_r) > \ell$. By Lemma \ref{usefullem} with $c = \ell$, $Q$ must contain a horizontal segment of length $\ell$ with closed endpoints. Note that we can do the same process with any entries in the same row of $B$ that contain ones in the copy of $S(M_{P'}, 1)$. This yields a set of horizontal segments of length $\ell$ with a distance of at least $c$ between consecutive rows and a distance of at least $c$ between columns containing the closest endpoints of consecutive segments, since $B$ contains $S(M_{P'}, 1)$ and each entry of $B$ corresponds to a $c \times c$ square. However this means that $S$ contains $P$, which is a contradiction.

Now we know that the number of ones in $B$ is at most  $\ex(\lceil \frac{n}{c} \rceil , S(M_{P'},1)) = O(n)$ by Lemma \ref{tardos_blank} and \cite{gduluth}, where the constant in the bound depends on $c$. For every entry $b_{i j} = 1$ in $B$, we define chunks which cover the elements of $S$. If $b_{i j}$ is not the rightmost entry in row $i$ with a $1$, then let $b_{i k}$ be the next entry in row $i$ after $b_{i j}$ for which $b_{i k} = 1$. The chunk $C_{i j}$ consists of all squares in the $i^{th}$ row of the grid between $U_{i j}$ and $U_{i k}$, including $U_{i j}$ but not $U_{i k}$. If $b_{i j}$ is the rightmost entry in row $i$ of $B$ with a $1$, then the chunk $C_{i j}$ consists of all squares $U_{i k}$ for $k \ge j$.

Consider any chunk $Q = C_{i j}$ of $S$. By definition, for every row $r$ in $Q$, we must have $\mu_1(Q_r) \le \ell$. Without loss of generality, we may assume that $r \in [0, c]$. Thus by Fubini's theorem, $\mu_2(Q) = \int_{r \in [0,c]} \mu_1(Q_r) d \mu_1(r) \le c \ell$. Thus $\mu_2(S) \le c \ell \ex(\lceil \frac{n}{c} \rceil , S(M_{P'},1)) = O(n)$.
\end{proof}

\section{Subsets of $\mathbb{R}^d$}\label{higher_d}

In this section, we generalize the extremal function $\px(n, P)$ and our results to higher dimensional sets of points. We start by defining extremal functions of forbidden $d$-dimensional 0-1 matrices. We say that a $d$-dimensional 0-1 matrix $A$ \emph{contains} a $d$-dimensional 0-1 matrix $B$ if some submatrix of $A$ is either equal to $B$ or can be turned into $B$ by changing some ones to zeroes. Otherwise we say that $A$ \emph{avoids} $B$, and that $A$ is \emph{$B$-free}. We define the $d$-dimensional extremal function $\ex(n, M, d)$ as the maximum possible number of ones in a $d$-dimensional $M$-free 0-1 matrix with all dimensions equal to $n$. Note that the extremal function $\ex(n, M, d)$ is the same as $\ex(n, M)$ when $d = 2$ and $M$ is a $2$-dimensional 0-1 matrix. The vast majority of research on $\ex(n, M, d)$ has been for the case $d = 2$, but several papers have focused on higher dimensions including \cite{KM, GTmat, ff, gkst, zhang}. 

As with matrix extremal functions, we can generalize $\px(n, P)$ from $\mathbb{R}^2$ to $\mathbb{R}^d$. Suppose that $P$ and $S$ are both subsets of $\mathbb{R}^d$. Let $C_i(P)$ be the set of all $i^{th}$ coordinates of points in $P$. Similarly let $C_i(S)$ be the set of all $i^{th}$ coordinates of points in $S$. We say that $S$ \emph{contains} $P$ if there exist functions $f_i: C_i(P) \rightarrow C_i(S)$ for each $i = 1, \dots, d$ such that all statements below are true: 
\begin{enumerate}
\item For each $i = 1, \dots, d$, we have $\frac{f_i(x)-f_i(x')}{x-x'} \ge 1$ for all $x, x' \in C_i(P)$ with $x > x'$.
\item  For all $(x_1, \dots, x_d) \in P$, we have $(f_1(x_1), \dots, f_d(x_d)) \in S$.
\end{enumerate}

If $S$ does not contain $P$, then $S$ \emph{avoids} $P$, i.e. $S$ is \emph{$P$-free}. For any  subset $P \subseteq \mathbb{R}^d$ and $n \in \mathbb{R}^{+}$, define $\px(n, P,d)$ as the supremum of $\mu_d(S)$ over all open $P$-free subsets $S \subseteq [0, n]^d$. 

We start with two simple observations about $\px(n, P, d)$ which generalize some earlier observations. 

\begin{lem}\label{dpxcontain}
If $P$ and $Q$ are subsets of $\mathbb{R}^d$ for which $Q$ contains $P$, then $\px(n,P,d) \le \px(n, Q,d)$.
\end{lem}

\begin{lem}
For all $m \le n$, we have $\px(m, P, d) \le \px(n, P, d)$.
\end{lem}

As in the $2$-dimensional case, we have that $\px(n, P, d)$ is continuous in $n$ for $n > 0$.

\begin{lem}\label{sandwichd}
For all $n > 0$ and all $\epsilon$ with $0 < \epsilon < n$, we have $\px(n, P, d) \le \px(n+\epsilon, P, d) < \px(n, P, d)+d (2n)^{d-1} \epsilon$.
\end{lem}

\begin{prop}
The function $\px(n, P, d)$ is continuous in $n$ for $n > 0$.
\end{prop}

\begin{proof}
Fix $n_0 > 0$ and $\epsilon > 0$. By Lemma \ref{sandwichd}, for $\delta = \min(\frac{\epsilon}{d (2n_0)^{d-1}}, \frac{n_0}{3})$ and $n < n_0+\delta$ we have $\px(n, P, d) \le \px(n_0+\delta, P, d) < \px(n_0, P, d)+\epsilon$. Moreover if $n > n_0 - \delta$, then we have $\px(n, P, d) \ge \px(n_0 - \delta, P, d) > \px(n_0, P, d)-d(2(n_0-\delta))^{d-1} \delta > \px(n_0, P, d)-\epsilon$. Thus we have shown that for all $n_0 > 0$ and for all $\epsilon > 0$ there exists $\delta > 0$ for which $|\px(n, P, d)- \px(n_0, P, d)| < \epsilon$ for all $n$ such that $|n-n_0| < \delta$.
\end{proof}

Given a subset $P \subseteq \mathbb{R}^{d+1}$, let $\Pr(P)$ denote the subset of $\mathbb{R}^d$ for which $(x_1,\dots, x_d) \in \Pr(P)$ if and only if there exists $x_{d+1}$ such that $(x_1,\dots, x_d, x_{d+1}) \in P$. The next observation is analogous to a fact about projections of $(d+1)$-dimensional 0-1 matrices into $d$ dimensions.

\begin{lem}\label{dproj}
For all $P \subseteq \mathbb{R}^{d+1}$, we have $\px(n, P, d+1) \ge n \px(n, \Pr(P), d)$.
\end{lem}

\begin{proof}
Let $S_t \subseteq [0, n]^d$ for $t \in \mathbb{N}$ be a sequence of $\Pr(P)$-free open subsets with the property that $\lim_{t \rightarrow \infty} \mu_d(S_t) = \px(n, \Pr(P), d)$. Note that since $S_t$ is open for each $t$, we have $S_t \subseteq (0, n)^d$. For each $t \in \mathbb{N}$, let $Z_t \subseteq (0, n)^{d+1}$ be the set such that  for all $x_1, \dots, x_{d+1} \in (0, n)$ we have $(x_1, \dots, x_d, x_{d+1}) \in Z_t$ if and only if $(x_1, \dots, x_d) \in S_t$. Then $Z_t$ is open, since it is the Cartesian product of open sets. Also $Z_t$ is $P$-free, or else $S_t$ would contain $\Pr(P)$. Thus $\px(n, P, d+1) \ge \lim_{t \rightarrow \infty} \mu_{d+1}(Z_t) = \lim_{t \rightarrow \infty} n \mu_{d}(S_t) = n  \px(n, \Pr(P), d)$.
\end{proof}

It is easy to find examples for which the last result is sharp. For example, let $P$ be any subset of $\mathbb{R}^{d+1}$ such that all elements of $P$ have the same last coordinate. 

\begin{prop}\label{lastcoord}
If all elements of $P \subseteq \mathbb{R}^{d+1}$ have the same last coordinate, then $\px(n, P, d+1) = n \px(n, \Pr(P), d)$.
\end{prop}

\begin{proof}
The lower bound $\px(n, P, d+1) \ge n \px(n, \Pr(P), d)$ follows from the last lemma. For the upper bound, let $S \subseteq [0, n]^{d+1}$ be open and $P$-free. For each $z \in [0, n]$, let $S_z$ be the subset of $\mathbb{R}^d$ such that $(x_1, \dots, x_d) \in S_z$ if and only if $(x_1, \dots, x_d, z) \in S$. 

Then $\mu_d(S_z) \le \px(n, \Pr(P), d)$ for all $z \in [0, n]$, or else $S$ would contain $P$. Since $S$ is open, the function $\mu_d(S_z): [0, n] \rightarrow \mathbb{R}$ is a measurable function by Fubini's theorem, so it is Lebesgue integrable on $[0, n]$. Thus $\mu_{d+1}(S) = \int_{z \in [0, n]} \mu_d(S_z) d\mu_1(z) \le n \px(n, \Pr(P), d)$.
\end{proof}

Using the last proposition, we can find many subsets $P \subseteq \mathbb{R}^d$ for which $\px(n, P, d) = \Theta(n^d)$ for $d > 2$.

\begin{cor}
If $P \subseteq \mathbb{R}^2$ is open, let $P' \subseteq \mathbb{R}^d$ for $d > 2$ be obtained from $P$ by letting $(x, y, z_1, \dots, z_{d-2}) \in P'$ if and only if $(x, y) \in P$ and $(z_1, \dots, z_{d-2}) = (0, \dots, 0)$. If $Q_{P'}$ denotes the set of points in $P'$ with rational coordinates, then $\px(n, Q_{P'}, d) = \Theta(n^d)$.
\end{cor}

\begin{proof}
This follows from Theorem \ref{open2} by applying Proposition \ref{lastcoord} a total of $d-2$ times to $Q_{P}$.
\end{proof}

By the last corollary and Lemma \ref{dpxcontain}, we also obtain $\px(n, P, d) = \Theta(n^d)$ for any open subset $P \subseteq \mathbb{R}^d$.

\begin{cor}
If $P \subseteq \mathbb{R}^d$ is open, then $\px(n, P, d) = \Theta(n^d)$ and $\px(n, Q_P, d) = \Theta(n^d)$.
\end{cor}

Given any finite subset $P \subseteq \mathbb{R}^d$, define $M_P$ to be the $d$-dimensional 0-1 matrix for which the length of the $i^{th}$ dimension of $M_P$ is the same as the number of distinct values for the $i^{th}$ coordinate of the points in $P$ for each $i$, the indices of the $i^{th}$ dimension of $M_P$ correspond to the values of the $i^{th}$ coordinate of the points in $P$ in order, and $M_P$ has a one in each entry corresponding to a point of $P$ and a zero in each other entry. 

The next lemma is proved analogously to Lemma \ref{tardos_blank}. Note for $d > 2$ that $S(M_P, k)$ is defined analogously to the definition for $d = 2$.

\begin{lem} 
For all finite subsets $P \subseteq \mathbb{R}^d$, $\ex(n, S(M_P,k),d) \le (k+1)^d \ex(\lceil \frac{n}{k+1} \rceil, M_P, d)$.
\end{lem}

As with the last lemma, the proof of the next lemma is analogous to the $2$-dimensional case in Lemma \ref{ps_blank}.

\begin{lem}
Suppose that $P \subseteq \mathbb{R}^d$ is a finite subset in which the differences between all consecutive values for each coordinate are in $[c_1, c_2]$. Let $P'$ be a dilation of $P$ by a factor of $q > 1$, i.e. for every point $(x_1, \dots, x_d) \in P$, the point $(qx_1, \dots, qx_d) \in P'$. Then $\px(n, P',d) \le (\lceil\frac{q c_2}{c_1}\rceil+1)^d \px(\lceil \frac{\lceil \frac{n}{c_1} \rceil}{\lceil\frac{q c_2}{c_1}\rceil+1} \rceil c_1, P, d)$.
\end{lem}

Using the last two lemmas, we obtain the next theorem with essentially the same proof method as in Theorem \ref{mainth}. This lets us derive corollaries about $\px(n, P, d)$ from the literature on $\ex(n, M, d)$.

\begin{thm}
For all finite subsets $P \subseteq \mathbb{R}^d$, $\px(n, P, d) = \Theta(\ex(n, M_P, d))$, where the constant in the bound depends on the differences between the consecutive values for each coordinate of $P$.
\end{thm}

Klazar and Marcus proved that $\ex(n, M, d) = O(n^{d-1})$ for every $d$-dimensional permutation matrix $M$ \cite{MT}. This bound was later sharpened in \cite{GTmat} where it was shown that $\ex(n, M, d) = 2^{O(k)}n^{d-1}$ for every $d$-dimensional permutation matrix $P$ with $k$ ones. This implies the corollary  below.

\begin{cor}
Let $P \subseteq \mathbb{R}^d$ with $|P| = k$ for which no points in $P$ have any common coordinates, and the differences between the consecutive values for each coordinate of $P$ are all at most $c$. Then $\px(n, P, d) = 2^{O(k)}n^{d-1}$, where the constant in the bound depends on $c$.
\end{cor}

In the same paper \cite{GTmat}, the authors proved a generalization of the K\H{o}vari-S\'{o}s-Tur\'{a}n theorem for $d$-dimensional 0-1 matrices by showing that $\ex(n, J_{k_1, \dots, k_d}, d) = O(n^{d-\alpha(k_1, \dots, k_d)})$, where $ J_{k_1, \dots, k_d}$ is the $d$-dimensional matrix of all ones with dimensions $k_1 \times \dots \times k_d$ and $\alpha(k_1, \dots, k_d) = \frac{\max(k_1, \dots, k_d)}{k_1 \dots k_d}$. As a result, for every $d$-dimensional 0-1 matrix $M$ there exists a constant $\epsilon > 0$ such that $\ex(n, M, d) = O(n^{d-\epsilon})$. This is because $M$ is contained in the $d$-dimensional all-ones matrix of the same dimensions.

\begin{cor}
For every finite subset $P \subseteq \mathbb{R}^d$ with $M_P =  J_{k_1, \dots, k_d}$, we have $\px(n, P,d) = O(n^{d-\alpha(k_1, \dots, k_d)})$.
\end{cor}

\begin{cor}
For every finite subset $P \subseteq \mathbb{R}^d$ there exists a constant $\epsilon > 0$ such that $\px(n, P, d) = O(n^{d-\epsilon})$.
\end{cor}

If $M$ is a $k \times k$ permutation matrix for $k \ge 2$, let $M_i$ be the $3$-dimensional 0-1 matrix of dimensions $k \times k \times i$ for which entry $(x_1, x_2, x_3)$ of $M_i$ is 1 if and only if entry $(x_1, x_2)$ of $M$ is 1. In \cite{gkst}, we proved a more general result which implies that $\ex(n, M_2, 3) = \Theta(n^{\frac{5}{2}})$ and $\ex(n, M_3, 3) = \Theta(n^{\frac{8}{3}})$. This implies the next corollary about $\px(n, P, d)$.

\begin{cor}
Let $P$ be a set of $k$ points in $\mathbb{R}^2$ with all coordinates distinct. Let $P_i$ be the set of $i k$ points in $\mathbb{R}^3$ which contains $(x, y, j)$ for each point $(x, y) \in P$ and each $j = 1, 2, \dots, i$. Then $\px(n, P_2, 3) = \Theta(n^{\frac{5}{2}})$ and $\px(n, P_3, 3) = \Theta(n^{\frac{8}{3}})$.
\end{cor}

We will next prove a much stronger result than the last corollary. It is a generalization of our strengthening of the K\H{o}vari-S\'{o}s-Tur\'{a}n theorem. For any $P \subseteq \mathbb{R}^d$ and for each integer $t \ge 2$ and real $c > 0$, let $Q_{P, t, c}$ denote the set of points $\left\{(x_1, \dots, x_d, y): (x_1, \dots, x_d) \in P \text{ and } y \in \left\{c, 2c, \dots, t c\right\}  \right\}$. For example, if $P \subseteq \mathbb{R}^2$ is a segment, then $Q_{P, 2, c}$ is a $3$-dimensional set of points that looks like an equal sign ($=$) and $Q_{P, 3, c}$ is a $3$-dimensional set of points that looks like an equivalence symbol ($\equiv$). 

The following proof is like the proof of Theorem \ref{kstcont} in Section \ref{infinite_p}, but more general. Part of the idea for this result is based on the main result of our paper \cite{gkst} for $d$-dimensional 0-1 matrices.

\begin{thm}\label{genkst}
For all $P \subseteq \mathbb{R}^d$ and fixed $t \ge 2$, $\px(n, Q_{P, t, c}, d+1) = O(\px(n, P, d)^{\frac{1}{t}} n^{d+1-\frac{d}{t}} + n^d)$, where the constants in the bound depend on $t$ and $c$.
\end{thm}

\begin{proof}
Let $S$ be an open $Q_{P, t, c}$-free subset of $[0,n]^{d+1}$. Let $S'$ be the $(t+d)$-dimensional set of points of the form $(x_1, \dots, x_d, y_1, \dots y_t)$ for which $(x_1, \dots, x_d, y_i) \in S$ for each $i = 1, 2, \dots, t$ and $y_i - y_{i+1} > c$ for each $i = 1, \dots, t-1$. First, we note that $S'$ is open. To see why this is true, define $T_1 = \left\{(x_1, \dots, x_d, y_1, \dots y_t) \in \mathbb{R}^{t+d}: (x_1, \dots, x_d, y_i) \in S \text{ for all } i = 1, \dots, t  \right\}$. Define $T_2 = \left\{(x_1, \dots, x_d, y_1, \dots y_t) \in \mathbb{R}^{t+d}: y_i - y_{i+1} > c \text{ for all } i = 1, 2, \dots, t-1 \right\}$. Clearly $T_2$ is an open subset of $\mathbb{R}^{t+d}$. Since $S$ is open, for every point $(x_1, \dots, x_d, y) \in S$ there exists $r > 0$ such that the open ball of radius $r$ centered at $(x_1, \dots, x_d, y)$ is a subset of $S$, i.e. $B((x_1, \dots, x_d, y), r) \subseteq S$. For each $(x_1, \dots, x_d,  y_1, \dots y_t) \in T_1$, we have $(x_1, \dots, x_d, y_i) \in S$ for all $i = 1, 2, \dots, t$, so there exist $r_i > 0$ for each $i = 1, 2, \dots, t$ such that $B((x_1, \dots, x_d, y_i), r_i) \subseteq S$ for all $i = 1, 2, \dots, t$. For every $(x_1', \dots, x_d', y'_1, \dots y'_t) \in B((x_1, \dots, x_d, y_1, \dots y_t), \min(r_1, r_2, \dots, r_t))$, we have $\sum_{j = 1}^d (x_j'-x_j)^2+\sum_{i = 1}^t (y'_i-y_i)^2 < \min(r_1^2, r_2^2, \dots, r_t^2)$, so $\sum_{j = 1}^d (x_j'-x_j)^2+(y'_i-y_i)^2 < r_i^2$ for all $i = 1, 2, \dots, t$, which implies that $(x_1', \dots, x_d', y'_i) \in S$ for all $i = 1, 2, \dots, t$, so $(x_1', \dots, x_d', y'_1, \dots y'_t) \in T_1$. Thus $B((x_1, \dots, x_d, y_1,\dots y_t), \min(r_1, r_2, \dots, r_t)) \subseteq T_1$, so $T_1$ is open. Hence $S'$ is also open, since $S' = T_1 \cap T_2$.

The points $(x_1, \dots, x_d, y_1, \dots y_t) \in S'$ must satisfy $0 < y_t < n-(t-1)c$ and $y_{i+1}+c < y_i < n-(i-1)c$ for all $i = 1, \dots, t-1$. Thus the projection of these points onto the last $t$ coordinates is a $t$-dimensional solid of volume $\Theta(n^t)$ by Lemma \ref{simplex}. For each fixed $y_1, \dots, y_t$ with $y_1, \dots, y_t \in (0, n)$, let $S''_{y_1, \dots, y_t}$ denote the set of points $(a_1, \dots, a_d, b_1, \dots, b_t) \in S'$ with $b_i = y_i$ for each $i$. Then $\mu_d(S''_{y_1, \dots, y_t}) \le \px(n, P, d)$ or else $S$ would contain $Q_{P, t}$. Since $S'$ is open, by Fubini's theorem the function $\mu_d(S''_{y_1, \dots, y_t}): [0, n]^t \rightarrow \mathbb{R}$ is a measurable function, so it is Lebesgue integrable on $[0, n]^t$. Thus $\mu_{t+d}(S') = \int_{(y_1, \dots, y_t) \in [0, n]^t} \mu_d(S''_{y_1, \dots, y_t}) d\mu_t(y_1, \dots, y_t) = O(\px(n, P, d) n^t)$.

For all $(x_1, \dots, x_d) \in [0, n]^d$, let $S'_{x_1, \dots, x_d}$ denote the set of points $(a_1, \dots, a_d, b_1,\dots, b_t) \in S'$ with $a_j = x_j$ for each $1 \le j \le d$, let $S_{x_1, \dots, x_d}$ denote the set of points $(a_1, \dots, a_d, b) \in S$ with $a_j = x_j$ for each $1 \le j \le d$, and let $m_{x_1, \dots, x_d} = \mu_1(S_{x_1, \dots, x_d})$. For $i < d$, let $S_{x_1, \dots, x_i}$ denote the set of points $(a_1, \dots, a_d, b) \in S$ with $a_j = x_j$ for each $1 \le j \le i$, and let $m_{x_1, \dots, x_i} = \mu_{d+1-i}(S_{x_1, \dots, x_i})$.

For all $x_1, \dots, x_d$ with $m_{x_1, \dots, x_d} \ge c t$, we must have $\mu_t(S'_{x_1, \dots, x_d}) \ge \frac{\max(m_{x_1, \dots, x_d}-c t,0)^t}{t^t}$. To see why this is true, suppose that $m_{x_1, \dots, x_d} \ge c t$. For each $(x_1, \dots, x_d, y) \in S_{x_1, \dots, x_d}$, let $Q_{x_1, \dots, x_d, y} = \left\{ (x_1, \dots, x_d, r): (x_1, \dots, x_d, r) \in S \text{ and } r \le y\right\}$. Note that 

\begin{align*}
\left\{(y_1, \dots, y_t): (\forall i \le t) ((x_1, \dots, x_d, y_i) \in S) \text{ and } (\forall i \le t-1) (\mu_1(Q_{x_1, \dots, x_d,y_{i+1}})+c < \mu_1(Q_{x_1, \dots, x_d,y_i})) \right\}     \subseteq  \\
   \left\{(y_1, \dots, y_t):  (\forall i \le t)  ((x_1, \dots, x_d, y_i) \in S) \text{ and }  (\forall i \le t-1) (y_{i+1}+c < y_i) \right\}.
\end{align*} 

\noindent Then we have $\mu_t(S'_{x_1, \dots, x_d}) = $

\footnotesize
\begin{align*}
\int_{y_t \in (0, n-(t-1)c)} \int_{y_{t-1} \in (y_t+c, n-(t-2)c)} \dots \int_{y_1 \in (y_2+c, n)} \prod_{i = 1}^{t} 1_{(x_1, \dots, x_d,y_i) \in S} d \mu_1(y_1) d \mu_1(y_2) \dots d \mu_1(y_t) \ge \\
\int_{y_t \in (0, n-(t-1)c)} \dots \int_{y_1 \in (y_2+c, n)} (\prod_{i = 1}^{t} 1_{(x_1, \dots, x_d,y_i) \in S}) (\prod_{i = 1}^{t-1} 1_{\mu_1(Q_{x_1, \dots, x_d,y_{i+1}})+c < \mu_1(Q_{x_1, \dots, x_d,y_i})})d \mu_1(y_1) d \mu_1(y_2) \dots d \mu_1(y_t).
\end{align*}
\normalsize

Given any $(y_1, \dots, y_t)$ with $(x_1, \dots, x_d, y_i) \in S$ for each $i \le t$, note that $\mu_1(Q_{x_1, \dots, x_d,y_i}) \in (0, m_{x_1, \dots, x_d})$ for each $i \le t$. Since $S$ is open, we can write $S_{x_1, \dots, x_d}$ as a countable union of intervals $I_j$ with open endpoints. Thus there are only countably many $b \in (0, m_{x_1, \dots, x_d})$ for which there does not exist $r$ such that $\mu_1(Q_{x_1, \dots, x_d,r}) = b$. Since countable sets have measure zero, the last integral is equal to

\begin{align*}
\int_{z_t \in (0, m_{x_1, \dots, x_d}-(t-1)c)} \int_{z_{t-1} \in (z_t+c, m_{x_1, \dots, x_d}-(t-2)c)} \dots \int_{z_1 \in (z_2+c, m_{x_1, \dots, x_d})} 1 d \mu_1(z_1) d \mu_1(z_2) \dots d \mu_1(z_t) \ge \\
 \frac{(m_{x_1, \dots, x_d}-c t)^t}{t^t}.
\end{align*}

\noindent where the last inequality is by Lemma \ref{simplex}. Thus by Fubini's theorem we have $\mu_{t+d}(S') = \int_{(x_1, \dots, x_d) \in [0, n]^d}\mu_t(S'_{x_1, \dots, x_d}) d\mu_d(x_1, \dots, x_d) \ge \int_{(x_1, \dots, x_d) \in [0, n]^d} \frac{\max((m_{x_1, \dots, x_d}-c t,0))^t}{t^t} d\mu_d(x_1, \dots, x_d)$.

Let $f: \mathbb{R} \rightarrow \mathbb{R}$ be defined by $f(z) =  \frac{\max(z-c t,0)^t}{t^t}$, so $f(m_{x_1, \dots, x_d}) =  \frac{\max(m_{x_1, \dots, x_d}-c t,0)^t}{t^t}$ and $f$ is convex. We can rewrite the last inequality in the last paragraph in the form $\mu_{t+d}(S') \ge \int_{(x_1, \dots, x_d) \in [0, n]^d} f(m_{x_1, \dots, x_d}) d\mu_d(x_1, \dots, x_d)$. With the inequality at the end of the second paragraph, we have $\int_{(x_1, \dots, x_d) \in [0, n]^d} f(m_{x_1, \dots, x_d}) d\mu_d(x_1, \dots, x_d) = O(\px(n, P, d) n^t)$. 

For each $j \le d-1$, by Jensen's inequality we have $n^{d-j} \int_{(x_1, \dots, x_j) \in [0, n]^j} f(\frac{m_{x_1, \dots, x_j}}{n^{d-j}}) d \mu_j(x_1, \dots, x_j) \le  n^{d-j-1} \int_{(x_1, \dots, x_{j}) \in [0, n]^{j}} \int_{x_{j+1} \in [0, n]} f(\frac{m_{x_1, \dots, x_{j+1}}}{n^{d-j-1}}) d \mu_1(x_{j+1}) d \mu_{j}(x_1, \dots, x_{j})$. Furthermore, the right side of this inequality equals  $n^{d-j-1} \int_{(x_1, \dots, x_{j+1}) \in [0, n]^{j+1}} f(\frac{m_{x_1, \dots, x_{j+1}}}{n^{d-j-1}}) d \mu_{j+1}(x_1, \dots, x_{j+1})$. So we have $n^{d-j} \int_{(x_1, \dots, x_j) \in [0, n]^j} f(\frac{m_{x_1, \dots, x_j}}{n^{d-j}}) d \mu_j(x_1, \dots, x_j) = O(\px(n, P, d) n^t)$ for all $j = 1, \dots, d-1$ by the inequality at the end of the last paragraph.

Thus by Jensen's inequality combined with the preceding inequality, we have $n^d f(\frac{\mu_{d+1}(S)}{n^d}) \le    n^{d-1} \int_{x_1 \in [0, n]} f(\frac{m_{x_1}}{n^{d-1}}) d \mu_1(x_1)     = O(\px(n, P, d)  n^t)$, so $n^d (\frac{\frac{\mu_{d+1}(S)}{n^d}}{t}-c)^t = O(\px(n, P, d) n^t)$ or else $\frac{\mu_{d+1}(S)}{n^d} \le c t$. Thus $\mu_{d+1}(S) = O(\px(n, P, d)^{\frac{1}{t}} n^{d+1-\frac{d}{t}}+n^d)$.
\end{proof}

As corollaries of Theorem \ref{genkst}, we derive sharp bounds on $\px(n, P, 3)$ for some bounded infinite subsets $P \subseteq \mathbb{R}^3$. We start with a corollary that can be applied to many bounded infinite subsets $P$.

\begin{cor}\label{master3d}
If $P \subseteq \mathbb{R}^2$ is any set with multiple points such that $\px(n, P) = O(n)$ and $c > 0$, then $\px(n, Q_{P,2, c},3) = \Theta(n^{\frac{5}{2}})$. If $P \subseteq \mathbb{R}^2$ is any set that has at least three points with different $x$-coordinates such that $\px(n, P) = O(n)$ and $c > 0$, then $\px(n, Q_{P,3, c},3) = \Theta(n^{\frac{8}{3}})$.
\end{cor}

\begin{proof}
For the first part, suppose that $P$ has multiple points. Without loss of generality, suppose that these points have $x$-coordinates that differ by $a$. Define $H_2$ to be the subset of $\mathbb{R}^2$ such that $(x, y) \in H_2$ if and only if there exists $z$ such that $(x, z, y) \in Q_{P,2,c}$. Then $H_2$ contains $\left\{(0, 0), (a, 0), (0, c), (a, c) \right\}$, so $\px(n, Q_{P,2,c},3) = \Omega(n^{\frac{5}{2}})$ by Lemma \ref{dproj}, Lemma \ref{pxcontain}, and Lemma \ref{kst2p}. The upper bound $\px(n, Q_{P,2,c},3) = O(n^{\frac{5}{2}})$ is from Theorem \ref{genkst} and the assumption that $\px(n, P) = O(n)$.

For the second part, suppose that $P$ has at least three points with different $x$-coordinates. Without loss of generality, suppose that each pair of these points have $x$-coordinates that differ by at least $a$.  Define $H_3$ to be the subset of $\mathbb{R}^2$ such that $(x, y) \in H_3$ if and only if there exists $z$ such that $(x, z, y) \in Q_{P,3, c}$. Then $H_3$ contains $\left\{(0, 0), (a,0), (2a, 0), (0, c), (a,c), (2a, c), (0, 2c), (a,2c), (2a, 2c) \right\}$, so $\px(n, Q_{P,3, c},3) = \Omega(n^{\frac{8}{3}})$ by Lemma \ref{dproj}, Lemma \ref{pxcontain}, and Lemma \ref{kst2p}. As with $Q_{P, 2, c}$, the upper bound $\px(n, Q_{P,3,c},3) = O(n^{\frac{8}{3}})$ is from Theorem \ref{genkst} and the assumption that $\px(n, P) = O(n)$.
\end{proof}

Now we apply Corollary \ref{master3d} to some specific forbidden subsets. The first subset is obtained from the one in Theorem \ref{autoarc}.

\begin{cor}
Suppose that $f: [0, a] \rightarrow \mathbb{R}$  is increasing, with $\frac{f(t)-f(s)}{t-s} \le b$ for all $s, t \in [0, a]$ with $s < t$. If $P$ is the set of points $\left\{(t,f(t)): t \in [0, a]\right\}$ and $c > 0$, then $\px(n, Q_{P,2,c},3) = \Theta(n^{\frac{5}{2}})$ and $\px(n, Q_{P,3,c},3) = \Theta(n^{\frac{8}{3}})$.
\end{cor}

\begin{proof}
This follows from Corollary \ref{master3d}, since $P$ contains at least $3$ points with different $x$-coordinates and $\px(n, P, 2) = \Theta(n)$ by Theorem \ref{autoarc}. 
\end{proof}

Next we apply Corollary \ref{master3d} to obtain sharp bounds on $Q_{P, 2, c}$ and $Q_{P, 3, c}$ when $P$ looks like a plus sign ($+$).

\begin{cor}
If $P$ is the set of points $\left\{(x, y): |x|, |y| \le d \text{ and } (x = 0 \text{ or } y = 0)\right\}$ and $c > 0$, then $\px(n, Q_{P,2,c},3) = \Theta(n^{\frac{5}{2}})$ and $\px(n, Q_{P,3,c},3) = \Theta(n^{\frac{8}{3}})$.
\end{cor}

\begin{proof}
This follows from Corollary \ref{master3d}, since $P$ contains at least $3$ points with different $x$-coordinates and $\px(n, P, 2) = \Theta(n)$ by Lemma \ref{addedseg}. 
\end{proof}

\section{Discussion}\label{openpr}

We proved that the extremal function $\px(n, P)$ fully encompasses the Zarankiewicz problem up to a constant factor, and more generally $\px(n, P)$ fully encompasses the matrix extremal function $\ex(n, M)$ up to a constant factor. Any past results on the Zarankiewicz problem or more generally on $\ex(n, M)$ imply corresponding results for $\px(n, P)$. While the inputs for $\ex(n, M)$ are discrete (a positive integer $n$ and a 0-1 matrix $M$), the inputs for $\px(n, P)$ form a continuum (a real number $n$ and any subset $P \subseteq \mathbb{R}^2$). 

We showed that $\px(n, P) = \Theta(n^2)$ for any open set $P$, where the constants in the lower bound depend on $P$. We proved the stronger result that $\px(n, Q_P) = \Theta(n^2)$, where $Q_P$ is the set of points with rational coordinates that are contained in $P$, but these bounds are only sharp up to a constant factor. A more general problem is to find sharper bounds on $\px(n, P, d)$ for open subsets $P \subseteq \mathbb{R}^d$. Although our bounds for open subsets $P$ are sharp up to a constant factor that depends on $P$, it is a natural problem to determine the exact value of $\px(n, P)$, particularly for certain open sets like open squares or open balls.

On the other hand, we found other bounded infinite subsets with extremal functions that behaved much like 0-1 matrix extremal functions. We showed that sets of points $E$ that look like equal signs ($=$) have $\px(n, E) = \Theta(n^{\frac{3}{2}})$, regardless of how long the horizontal segments are. On the other hand, sets of points $A$ that look like plus signs ($+$) have $\px(n, A) = \Theta(n)$ by Theorem \ref{horizseg} and Lemma \ref{addedseg}. There are other natural sets of points that we have not considered, such as sets of points in the shape of $\sqcup$ or H. One could also consider infinite sets of points in the shape of the boundary of a rectangle ($\square$) and other more complex grid-like sets of points.

All of the forbidden infinite subsets of $\mathbb{R}^2$ in the previous paragraph are unions of horizontal or vertical line segments, but some extend naturally to sets formed by diagonal segments. For example, if $P$ is the subset consisting of the segment with endpoints at $(0, 0)$ and $(1, 0)$ and the segment with endpoints at $(0, 1)$ and $(1, 2)$, what is $\px(n, P)$? In this case, note that $P$ contains the subset which consists of the four points $Q = \left\{(0,0), (1,0),(0,1),(1,2) \right\}$, and $\px(n, Q) = \Theta(n \log{n})$ by Theorem \ref{mainth} and \cite{tardos05}, so $\px(n, P) = \Omega(n \log{n})$.

\begin{conj}
If $P$ is the subset of $\mathbb{R}^2$ consisting of the segment with endpoints at $(0, 0)$ and $(1, 0)$ and the segment with endpoints at $(0, 1)$ and $(1, 2)$, then $\px(n, P) = \Theta(n \log{n})$.
\end{conj}

We also have a conjecture for any disjoint union of horizontal segments $P$ with no two points having the same $x$-coordinate, for which we conjecture a quasilinear upper bound on $\px(n, P)$. This is in analogue with the quasilinear upper bound on $\ex(n, M)$ when $M$ is a 0-1 matrix with no pair of ones in the same column \cite{Ke}.

\begin{conj}
If $P$ is a disjoint union of horizontal segments with no two points having the same $x$-coordinate, then there exists a constant $t$ such that $\px(n, P) = O(n 2^{\alpha(n)^t})$.
\end{conj}

Similar questions could also be asked for sets of points that are formed by arcs instead of only segments. We also wonder what is the supremum of $\px(n, P)$ over all subsets $P \subseteq [0, 1]^2$ for which all points in $P$ have different $x$-coordinates and all points in $P$ have different $y$-coordinates.

In Section \ref{basic_prop}, we found operations that can be performed on sets of points $P \subseteq \mathbb{R}^2$ to obtain a new set of points $P'$ such that $\px(n, P') = \Theta(\px(n, P))$. For example, if $P$ has a rightmost column containing some point $p$, then you can add a horizontal segment of length $c$ to $P$ with its left endpoint on $p$ to obtain a new subset of $\mathbb{R}^2$, and this only increases $\px(n, P)$ by at most $c n$. Also we showed that dilating \emph{finite} subsets $P$ only changes $\px(n, P)$ by at most a constant multiplicative factor. For a better understanding of $\px(n, P)$, it would be useful to find more operations. For example, Keszegh proved in \cite{Ke} that if $J$ is a 0-1 matrix that has a one in the bottom right corner and $K$ is a 0-1 matrix that has a one in the top left corner, and $R$ is obtained from translating $J$ and $K$ so that the bottom right corner of $J$ overlaps the top left corner of $K$ and filling the blank space with zeroes, then $\ex(n, R) \le \ex(n, J)+\ex(n, K)$. Is an analogous operation possible for subsets of $\mathbb{R}^2$?

\begin{conj}
Suppose that $P \subseteq \mathbb{R}^2$ contains the point $(c, d)$ and all other points in $P$ have $x$-coordinate at most $c$ and $y$-coordinate at most $d$. Suppose that $Q \subseteq \mathbb{R}^2$ contains the point $(c, d)$ and all other points in $Q$ have $x$-coordinate at least $c$ and $y$-coordinate at least $d$. Then $\px(n, P \cup Q) \le \px(n, P)+\px(n, Q)$.
\end{conj}

\subsection{When the size of $P$ grows with $n$}

All of the results in this paper focused on forbidden \emph{bounded} sets of points. Several papers have investigated $\ex(n, M)$ and corresponding extremal functions of forbidden sequences where the size of the forbidden 0-1 matrix $M$ and the size of the forbidden sequence are allowed to grow with respect to $n$ \cite{rs, wp, csd, gtform}. For example, we showed in \cite{csd} that if $M_s$ is any $2 \times s$ 0-1 matrix that has ones in both the first and last columns, then $\ex(n, M_s) = \Omega(n^{2-o(1)})$ if and only if $s(n) = \Omega(n^{1-o(1)})$, answering a question from \cite{wp}.

Analogously, it is natural to investigate $\px(n, P)$ where the size of the forbidden subset $P \subseteq \mathbb{R}^2$ is allowed to grow with respect to $n$. Let $Q_s$ be any set of points with minimum $x$-coordinate $0$, maximum $x$-coordinate $s$, and all $y$-coordinates equal to $0$ or $1$. Then $P_{s, 2}$ contains $Q_s$ with $P_{s,2}$ defined as in Theorem \ref{t=2}, so $\px(n, Q_s) \le \px(n, P_{s, 2})$. 

\begin{prop}
For all subsets $Q_s \subseteq \mathbb{R}^2$ with minimum $x$-coordinate $0$, maximum $x$-coordinate $s$, and all $y$-coordinates equal to $0$ or $1$, we have $\px(n, Q_s) = \Omega(n^{2-o(1)})$ if and only if $s(n) = \Omega(n^{1-o(1)})$.
\end{prop}

\begin{proof}
If $s = \Omega(n^{1-o(1)})$, then the open set $(0, s) \times (0, n)$ avoids $Q_s$, which shows that $\px(n, Q_s) = \Omega(n^{2-o(1)})$. 

If $s \neq \Omega(n^{1-o(1)})$, then there is a constant $\alpha < 1$ and an infinite sequence $x_1 < x_2 < \dots$ of positive reals such that $\lim_{i \rightarrow \infty} x_i = \infty$ and $s(x_i) < x_i^{\alpha}$ for all $i > 0$. If $s(n) < n^{\alpha}$ for some $\alpha < 1$, then we have $\px(n, Q_s) \le \px(n, P_{s, 2}) = O(n^{\frac{\alpha}{2}+\frac{3}{2}})$ by Theorem \ref{t=2}. Note that $\frac{\alpha}{2}+\frac{3}{2} < 2$, since $\alpha < 1$. Thus if $s \neq \Omega(n^{1-o(1)})$, then $\px(n, Q_s) \neq \Omega(n^{2-o(1)})$. Thus we have shown that $\px(n, Q_s) = \Omega(n^{2-o(1)})$ if and only if $s(n) = \Omega(n^{1-o(1)})$.
\end{proof}

\subsection{Saturation functions for subsets of $\mathbb{R}^2$}

Another potential direction for future research is saturation functions. Saturation problems have been studied for decades, with much of the focus on graphs \cite{dkm, ehm, fk0, kt}, posets \cite{fkk, klm}, and set systems \cite{fkk0, gkl}. Recently a saturation function was introduced for forbidden 0-1 matrices \cite{cb}. We say 0-1 matrix $A$ is \emph{saturating} for 0-1 matrix $M$ if $A$ avoids $M$ but any matrix obtained from $A$ by changing a zero to a one must contain $M$. Define $\sat(n, M)$ as the minimum number of ones in an $n \times n$ 0-1 matrix that is saturating for $M$. Fulek and Keszegh proved that $\sat(n, M) = O(1)$ or $\sat(n, M) = \Theta(n)$ for all 0-1 matrices $M$ \cite{fk}. The results in \cite{fk} showed that $\sat(n, M) = \Theta(n)$ for almost all $k \times k$ 0-1 matrices $M$, but infinite families of 0-1 matrices $M$ with $\sat(n, M) = O(1)$ were found in \cite{ben} and \cite{gsat}. In the same way that $\px(n, P)$ encompasses $\ex(n, M)$, it seems natural to investigate if there is a saturation function for subsets of $\mathbb{R}^2$ that encompasses $\sat(n, M)$. 

Say that an open subset $S \subseteq [0,n]^2$ is \emph{saturating} for $P \subseteq \mathbb{R}^2$ in $[0, n]^2$ if $S$ avoids $P$ but any open subset $S' \subseteq [0, n]^2$ that is a proper superset of $S$ must contain $P$. So we can define $\osat(n, P)$ as the infimum of $\mu_2(S)$ over all open subsets $S \subseteq [0, n]^2$ that are saturating for $P$ in $[0, n]^2$. Note that $\osat(n, P)$ is only defined if there exists an open subet of $[0, n]^2$ that is saturating for $P$  in $[0, n]^2$. Clearly $\osat(n, P) \le \px(n, P)$ whenever $\osat(n, P)$ is defined, and this is sharp. We can see that this is sharp for the simplest example of nonempty $P \subseteq \mathbb{R}^2$.

\begin{lem}
If $P$ consists of a single point, then $\osat(n, P) = 0$.
\end{lem}

\begin{proof}
For all $n > 0$, the empty set is saturating for $P$ in $[0, n]^2$.
\end{proof}

Like the 0-1 matrix saturation function $\sat(n, M)$, we conjecture that the subset saturation function $\osat(n, P)$ is either $O(1)$ or $\Theta(n)$ for all subsets $P \subseteq \mathbb{R}^2$ on which $\osat(n, P)$ is defined.

\begin{conj}
For any subset $P \subseteq \mathbb{R}^2$ such that $\osat(n, P)$ is defined, either $\osat(n, P) = O(1)$ or $\osat(n, P) = \Theta(n)$, where the constants in the bounds depend on $P$.
\end{conj}

We conjecture that $\osat(n, P)$ is defined for every finite subset $P \subseteq \mathbb{R}^2$, and that there is a relationship between $\osat(n, P)$ and $\sat(n, M_P)$ for finite subsets $P$ that is analogous to the relationship between $\px(n, P)$ and $\ex(n, M_P)$ in Theorem \ref{mainth}.

\begin{conj}
For every finite subset $P \subseteq \mathbb{R}^2$, the saturation function $\osat(n, P)$ is defined and we have $\osat(n, P) = \Theta(\sat(n, M_P))$.
\end{conj}

A different saturation function for 0-1 matrices was investigated in \cite{dpt}, it would also be natural to see if there is some saturation function for subsets of $\mathbb{R}^2$ that encompasses the function in \cite{dpt}. Another possible research direction is to investigate other natural variants of the definitions for $\px(n, P)$ and $\osat(n, P)$. In our definition of $\px(n, P)$, we required the $P$-free subsets $S \subseteq [0, n]^2$ to be open. This allowed us to show that $\px(n, P) = \Theta(\ex(n, M_P))$, along with many other results. It would be interesting to see how the results in this paper change if we require the $P$-free subsets $S$ to be closed instead. Alternatively we could require the sets $S$ to be Borel sets, or to be Lebesgue measurable sets in the definitions of $\px(n, P)$ and $\osat(n, P)$.

It would also make sense to investigate $\px(n, P)$ and $\osat(n, P)$ over all open $P$-free subsets $S \subseteq [0, m] \times [0, n]$. This is in analogue with the 0-1 matrix extremal function $\ex(n, M)$, which has a more general variant $\ex(m, n, M)$ that maximizes the number of ones among all $m \times n$ $M$-free 0-1 matrices. Other regions besides rectangles could also be considered for replacing $[0, n]^2$, such as balls of radius $n$.

\end{document}